\documentclass[a4paper,12pt]{article}
\usepackage{amssymb,amsmath,amsthm}
\usepackage{marvosym}
\usepackage{graphicx}

\usepackage{amsfonts}
\usepackage{latexsym}
\usepackage{color}
\usepackage[english]{babel}
\usepackage{hyperref}
\usepackage{enumitem}
\usepackage{float}
\usepackage{caption}
\usepackage{subcaption}
\usepackage[letterpaper,margin=1.2in,bottom=1.5in,top=1.5in]{geometry}

\numberwithin{equation}{section}

\newtheorem{lemma}{Lemma}[section]
\newtheorem{theorem}[lemma]{Theorem}

\newtheorem{corollary}[lemma]{Corollary}
\newtheorem{remark}[lemma]{Remark}

\theoremstyle{definition}
\newtheorem{manualtheoreminner}{Assumption}
\newenvironment{assumption}[1]{%
	\renewcommand\themanualtheoreminner{#1}%
	\manualtheoreminner
}{\endmanualtheoreminner}
\theoremstyle{plain}

\newtheorem{definition}[lemma]{Definition}

\newcommand{\ta}{{a}}
\newcommand{\tw}{{\tt w}}
\newcommand{\tf}{{\tt f}}
\newcommand{\tr}{{\tt r}}
\newcommand{\tb}{{\tt b}}

\newcommand{\R}{\mathbb R}

\newcommand{\Z}{\mathbb Z}

\newcommand{\N}{\mathbb N}
\newcommand{\T}{\mathbb T}

\newcommand{\ii }{{\rm i} }
\newcommand{\vphi}{\varphi}
\def\norma#1{\|#1  \|}
\def\norm#1{\|#1  \|}
\def\rA#1{\cA_{#1}}
\def\rS#1{\cS_{#1}}
\def\im{{\rm i}}
\def\vf{\varphi}
\def\varep{\varepsilon}
\def\ep{\varepsilon}

\def\timereg{C^{\infty}_b}
\def\cinfinito#1{C^{\infty}_b(\R;\rA\rho^{#1})}
\def\cinfinitos#1#2{C^{\infty}_b(\R;\rS{#1}^{#2})}
\def\semi#1#2#3{\wp_{\rho,#2}^{#1}\left(#3\right)}
\def\norfou#1#2#3{\semi{#1}{#2,N}{#3}}
\def\norfs#1#2#3{\norfou{#2}{#3}{#1}}
\def\rar{\cS_\rho}
\def\td{{\tt d}}

\newcommand{\csi}{\xi}

\def\cZ{{\mathcal{Z}}}
\def\cJ{{\mathcal{J}}}

\def\tk{{\tt k}}

\definecolor{awesome}{rgb}{1.0, 0.13, 0.32}
\definecolor{darkgr}{rgb}{0.0, 0.62, 0.42}
\definecolor{cyan}{rgb}{0.0, 0.72, 0.92}

\def\SF#1{\mathcal{S}_{\rho}^{#1}}

\def\uno{{\bf 1}}
\def\tM{{\tt M}}

\def\cF{{\mathcal{F}}}
\def\cH{{\mathcal{H}}}

\def\QS{{\mathcal{Q}\kern-0.3pt\mathcal{S}}}
\def\cC{{\mathcal{C}}}

\def\cS{{\mathcal{S}}}
\def\cD{{\mathcal{D}}}
\def\tk{{\tt k}}

\def\tC{{\tt C}}
  
\def\tB{{\tt B}}
\def\tD{{\tt D}}
\def\tR{{\tt R}}

\def\cA{{\mathcal{A}}}
\def\cU{{\mathcal{U}}}

\def\cB{{\mathcal{B}}}
\def\tb{{\tt b}}

\newcommand{\Span}{\mathrm{span}}
\newcommand{\scala}[2]{{#1} \cdot {#2} }
\newcommand{\Momega}{{M_{a}}}
\def\mom{{\tt M}}
\def\chir{{\chi^{(\tR)}}}
\def\da{{\frak{a}}}
\DeclareMathOperator\Rep{Rep}
\DeclareMathOperator{\OPS}{OPS}
\newcommand{\res}{\mathrm{res}}
\newcommand{\nr}{\mathrm{nr}}
\newcommand{\Sm}{\mathrm{S}}

\begin{document}

\title{{\bf  Growth of Sobolev norms in quasi-integrable quantum systems  \vspace{15pt}\\ Croissance des normes de Sobolev dans les syst\`{e}mes quantiques quasi-int\'{e}grables} }

\date{}


\author{ Dario Bambusi\footnote{Dipartimento di Matematica, Universit\`a degli Studi di Milano, Via Saldini 50, I-20133
Milano. 
 \textit{Email: } \texttt{dario.bambusi@unimi.it}}, Beatrice Langella\footnote{International School for Advanced Studies (SISSA), via Bonomea 265, I-34136 Trieste.
\textit{Email: } \texttt{beatrice.langella@sissa.it}}.
 }

\maketitle

\begin{abstract}
We prove an abstract result giving a $ \langle t\rangle^{\ep}$
upper bound on the growth of the Sobolev norms of a time-dependent
Schr\"odinger equation of the form $\im \dot \psi = H_0 \psi + V(t)
\psi $. {Here} $H_0$ is assumed to be the Hamiltonian of a steep quantum
integrable system and to be a {pseudodifferential} operator of order
$\td>1$; $V(t)$ is a time-dependent family of pseudodifferential
operators, unbounded, but of order $\tb<\td$. The abstract theorem is
then applied to {perturbations of the quantum anharmonic oscillators
  in dimension 2 and to perturbations of the Laplacian on a manifold
  with integrable geodesic flow, and in particular Zoll manifolds,
  rotation-invariant surfaces and Lie groups.} The proof is based on a
quantum version of the proof of the classical Nekhoroshev theorem.
 
 \begin{center}
 	\textbf{R\'{e}sum\'{e}}
 \end{center}

Nous prouvons un r\'{e}sultat abstrait donnant une majoration de la forme
$\langle t\rangle^{\ep}$ pour la croissance des normes de Sobolev
d'une \'{e}quation de Schrödinger d\'{e}pendant du temps de la forme $\im
\dot{\psi} = H_0 \psi + V(t) \psi$. On suppose que $H_0$ est
l'hamiltonien d'un syst\`{e}me quantique int\'{e}grable {escarp\'{e}} (steep) et qu'il est un op\'{e}rateur pseudo-diff\'{e}rentiel d'ordre $\td > 1$ ; $V(t)$ est une famille d\'{e}pendant du temps d'op\'{e}rateurs pseudo-diff\'{e}rentiels, non born\'{e}s, mais d'ordre $\tb < \td$. Le th\'{e}or\`{e}me abstrait est ensuite appliqu\'{e} aux perturbations des oscillateurs quantiques anharmoniques en dimension 2 et aux perturbations du laplacien sur une vari\'{e}t\'{e} avec flot g\'{e}od\'{e}sique int\'{e}grable, et en particulier sur les vari\'{e}t\'{e}s de Zoll, les surfaces invariantes par rotation et les groupes de Lie. La d\'{e}monstration repose sur une version quantique de la preuve du th\'{e}or\`{e}me de Nekhoroshev classique.
\end{abstract}
\noindent


\medskip



\setcounter{tocdepth}{1}
\tableofcontents
 
\section{Introduction}\label{intro}

In this paper we prove an abstract theorem giving a $\langle
t\rangle^{\ep}$ upper bound for the Sobolev norms {of the
  solutions} of an abstract
Schr\"odinger equation of the form
\begin{equation}
  \label{schro}
\im\dot \psi=(H_0+V(t))\psi\ ,\quad \psi\in\cH
\end{equation}
where $\cH$ is a Hilbert space, and $H_0$ is the Hamiltonian of a
quantum system which is {globally} integrable in a sense {defined}
below {and steep}; $H_0$
is also assumed to be a pseudodifferential operator of order $\td>1$,
and $V(t)$ is a smooth time-dependent family of self-adjoint
pseudodifferential operators of order $\tb<\td$. {In this sense $V$ can be considered a perturbation of $H_0$ and the system $H_0 + V(t)$ can be called \emph{quasi-integrable}.} 

{The novelty of this result is twofold: {it} extends the class of
  unperturbed systems {for which upper bounds of the Sobolev
    norms can be obtained, {and} it allows to treat the case of
    unbounded perturbations}. We emphasize that the class of unperturbed systems treated
  here strictly contains
  \emph{all} the systems of order $\td>1$ for which estimates on the
  growth of Sobolev norms have been obtained; it also contains new
   systems, like the quantum 2-d
  anharmonic oscillator and Lie groups.

\vskip 10 pt
We now describe more in detail our result. First, following
\cite{BGMR} (see also \cite{Fischer}) we
introduce some abstract algebras of linear operators in $\cH$ that enjoy
the 
properties typical of pseudodifferential operators. {We actually call
the elements of these algebras pseudodifferential operators.}

Then we give the definition of globally integrable quantum system. The
idea is to introduce some operators $A_j$, $j=1,...,d$ that are the
quantum analogue of the classical action variables and to consider
quantum systems with Hamiltonian $H_0=h_0(A_1,...,A_d)$, namely
Hamiltonians which are functions of the quantum action variables. In
turn, the quantum action variables are defined to be $d$ commuting
self-adjoint pseudodifferential operators of order 1, whose joint
  spectrum is contained in $\Z^d+\kappa$, with some $\kappa \in
  \R^d$.  The motivation of this definition rests in the classical
results by Duistermaat--Guillemin \cite{DG75}, Colin de Verdi\`{e}re
\cite{CdV2}, and Helffer--Robert \cite{HR82}, which ensure that, under
suitable assumptions, the quantization of a classical action variable
is a perturbation of a pseudodifferential operator with spectrum
contained in $\Z+\kappa$ with $\kappa\in\R$.  We point out that our
definition, which does not involve directly the action variables of
the classical system, is quite flexible, since it {applies also to
  systems whose classical action} variables are poorly known and to
quantization of some superintegrable systems \cite{Nek0,Fas} in which
the Hamiltonian is independent of some of the action
variables. {Finally, we assume that the function $h_0$ is homogeneous
  of degree $\td>1$ and fulfills the steepness assumption of the
  classical Nekhoroshev's theorem (see Definition \ref{steep.def}
  below).}
{Concerning the perturbation $V$, we assume that  it is a pseudodifferential operator of order $\tb < \td$. 
In the case $\tb = \td$, $V$ cannot be considered as a perturbation of $H_0$. For this reason, we do not think it can be treated within our framework.}

We come now to the applications of the abstract result. The first application is to perturbations of a 2-dimensional
quantum anharmonic oscillator with Hamiltonian
\begin{equation}
  \label{qana}
H_0=-\frac{\Delta}{2}+\frac{\|x\|^{2\ell}}{2\ell}\ ,\quad x\in\R^2\ ,\quad \ell\in \N\ ,\quad
\ell\geq 2\ ,
\end{equation}
for which we prove the $\langle t\rangle^\ep$ upper bound on the
growth of Sobolev norms. We recall that for {anharmonic oscillators}
the situation was clear only for the 1-d case \cite{BGMR}, while for
the higher dimensional case only {an estimate by} $\langle t\rangle^s$
for the $\cH^s$ norm was known for the case of bounded
  perturbations (see \cite{MaRo}), {while, for unbounded
    perturbations, an upper bound of the form $\langle
    t\rangle^{s\alpha_s}$, with an exponent $\alpha_s>1$ which
    diverges as $\tb \rightarrow \td$, was proven in \cite{MaRo}.\footnote{{Actually the result of \cite{MaRo} applies to very general systems and does not rely on a pseudo-differential setting, but typically only allows to prove estimates with the exponent of $\langle t \rangle$ which depends on $s$}.}}. A
$\langle t\rangle^\ep$ estimate was out of reach {for the 2d case} with
previous methods. 
	{The second application is to Schr\"odinger equations on
          {compact} manifolds
  with  {globally} integrable geodesic flow,  {namely manifolds in
  which the Hamiltonian of the geodesic flow is integrable, and
  furthermore the action variables are globally defined and the Hamiltonian is steep.} In this paper, in order to be determined, we present some
  specific examples.} {First (i)} we recover the known
  results for tori \cite{Bourgain,delort,berti_maspero,BLMgrowth} and
  Zoll Manifolds \cite{BGMR}, then (ii) we consider rotation invariant
  surfaces (following \cite{CdV2,delort}) and construct the operators
  $A_j$ by quantizing the classical action variables. The novelty of
  {the} result we get {for rotation invariant surfaces} is that we can deal with \emph{unbounded}
  perturbations of the Laplacian. Another example (iii) is that of
  the Schr\"odinger equation on a Lie group ({see
  \cite{BerPro11,BerCorPro} for a KAM type result}): it is known that the geodesic flow on a
  compact Lie group is integrable (\cite{M82,Bolsinov}), but very
  little is known on the action angle variables. For this reason we
  work directly at a quantum level. To this end we use the
  intrinsic pseudodifferential calculus on Lie groups developed in 
  \cite{Fischer,ruzhansky_turunen} and construct directly the quantum
  actions in the case of compact, simply connected Lie groups. Here
  the lattice $\Z^d$ is essentially the lattice of the {dominant}
  weights of the irreducible representations of the Lie group. {The
  result controlling growth of Sobolev norms for solutions of the Schr\"odinger 
  equation on Lie groups is new.}

\vskip10pt

We come now to a description of the proof of the abstract result
{(namely Theorem \ref{main})}. The
present paper is a direct continuation of the works
\cite{BLMnr,BLMres,BLMgrowth,anarmonico}, and is based on the quantization of
the proof of the classical Nekhoroshev theorem. Here in particular we
develop an abstract version of the proof, which is based just on the use
of the lattice of the joint spectrum of the actions.

Precisely, the proof consists of two
steps: first one {uses the methods of normal form to construct
  iteratively} a family of unitary time-dependent operators
conjugating the original Hamiltonian \eqref{schro} to a Hamiltonian of
the form
\begin{equation}
  \label{nor.forma}
H_0+Z(t)+R(t)
\end{equation}
with $Z(t)$ which is a ``normal form'' operator (see Definition
\ref{def.nf} below) and $R(t)$ a smoothing
operator, which plays the role of a remainder. This was done in
\cite{BLMres,BLMgrowth} for the Schr\"odinger operator on $\T^d$ by quantizing the
classical normal form procedure. {We remark that the use of pseudo-differential calculus is what allows in \cite{BLMres, BLMgrowth} to deal with unbounded perturbations.} The second step of the
proof {in \cite{BLMres,BLMgrowth}} consists in {analyzing the
  structure of the normal form operator, namely in studying the way it couples
  different Fourier modes depending on the resonance relations
  fulfilled by the frequencies of the classical system. This is
  essentially a quantum version of the geometric construction of
Nekhoroshev theorem (see e.g. \cite{Nek1, Nek2,GCB,nek_noi} for the
classical construction). As a result we obtain that  $Z(t)$ has a block diagonal
structure with dyadic blocks}, so that for the
dynamics of $H_0+Z(t)$ the Sobolev norms remain bounded forever. The
addition of the remainder $R(t)$ is the responsible for the $\langle
t\rangle^{\ep}$ estimate on the growth.

To realize an abstract version of the above construction, one has to
develop several new tools. {A major difference with respect to the
  works \cite{BLMnr, BLMres, BLMgrowth, anarmonico} is that we use
  here directly pseudo-differential operators, never dealing with the corresponding symbols. This enables to work only with the quantum actions, instead of the classical ones.} {In turn, the quantum actions are used to define a Fourier expansion of pseudo-differential operators, in which the Fourier coefficients are labeled by the points of the lattice of their joint
spectrum. This is the main new technical {ingredient} of the present paper.}
	
{We point out that the geometric part} of
the proof is more complicated here than in \cite{BLMres,BLMgrowth}
because we deal here with the steep case, while in
\cite{BLMres,BLMgrowth} only the convex case was treated.

\vskip10pt

{We conclude this introduction by adding some further comments on
  the connection of the present result with previous ones.}
 Our main references is the paper \cite{Bourgain} (see
  also \cite{delort,berti_maspero}), in which Bourgain deals with the
  Schr\"odinger equation on the torus. Bourgain's approach {is based
      on a dyadic decomposition of $\Z^d$ which is almost invariant
      {for $H$} and only allows to deal with \emph{bounded} perturbations of
    $-\Delta$ \emph{on tori}.
    Bourgain's result was extended
    to the case of unbounded perturbations in \cite{BLMgrowth}, which
    in turn is the starting point of the present work. {The main new
    point of \cite{BLMgrowth} is that it
    gives a decomposition of $\Z^d$ analogous to Bourgain's one, but based on the resonance properties of the 
    frequencies of the underlying classical system. Such resonance properties admit a natural generalization to more general systems, and this is what is developed in
    the present paper.} 

The other direct reference for our work is \cite{BGMR} (which in turn
is closely related to \cite{1,bam_17b,BM3,MaRo,BGMR1}). In \cite{BGMR}
essentially two cases were considered: systems with periodic classical
flows
(1-d systems and Zoll manifolds) and systems of degree $\td=1$ (half
wave equation in dimension 1 and harmonic oscillators). The main
novelty of the present approach is that we are able to deal with the
case of quantization of ``general'' classical integrable systems, in
which the flow is either periodic or quasiperiodic depending on the initial
datum. However we restrict here to the case $\td>1$.

Concerning the case $\td=1$, we remark that known upper bounds on the
growth of Sobolev norms just pertain the case of a {perturbation
  $V(t)$ which depends quasiperiodically} on time, with frequencies
which are nonresonant with the asymptotic of the gaps between
eigenvalues of the unperturbed operator $H_0$. Still for the case
$\td=1$ some interesting counterexamples, showing that growth of
Sobolev norms actually occurs, and furthermore in the resonant case is
a quite general phenomenon, have been obtained \cite{Del14, BGMR,
  Mas19,Mas2021, Mas2022,Thom}. We plan to investigate possible extensions of
our method to the case $\td=1$ in the future.

\vskip10pt
\noindent {\it Acknowledgments}. During the preliminary work which led
to this result we had many discussions on integrable systems and
Fourier integral operators, which were essential in order to
understand the theory of \cite{CdV2} and to arrive to our {definition}
of quantum integrable system: we warmly thank Francesco Fass\`o,
Didier Robert and San V\~{u} Ng\d{o}c. We also thank Bert Van Geemen
and Michela Procesi for some discussions and suggestions on Lie
groups, {Santiago Barbieri and Laurent Niederman for discussions about
  steepness, and Massimiliano Berti, Beno\^{i}t Gr\'{e}bert, Alberto
  Maspero and Riccardo Montalto for comments on a first version of
  this paper.}\\ We acknowledge the support of GNFM and of  
the research project PRIN 2020XBFL ``Hamiltonian and dispersive PDEs''.

\vskip 20 pt
\null

\part{Statements}

\section{Main results}
\subsection{Abstract pseudodifferential operators}\label{sec:alg}
Let $\cH$ be a Hilbert space and $K_0$ a self-adjoint positive operator
with compact inverse.

We define a scale of Hilbert spaces by $\cH^r:=\operatorname{dom}(K_0^r)$ (the domain
of the operator $K_0^r$) if $r\geq 0$, and $\cH^{r}:=(\cH^{-r})^\prime$
(the dual space {w.r.t. the scalar product of $\cH$}) if $ r<0$.  We denote by $\cH^{-\infty} =
\bigcup_{r\in\R}\cH^r$ and $\cH^{+\infty} = \bigcap_{r\in\R}\cH^r$.
We endow $\cH^r$ with the natural norm $\norm{\psi}_r:= \norm{(K_0)^r
	\psi}_{0}$, where $\norm{\cdot}_0$ is the norm of $\cH^0 \equiv
\cH$. From the spectral decomposition of $K_0$ it follows that for any
$m\in\R$, $\cH^{+\infty}$ is a dense linear subspace of $\cH^m$.
In the following, when we say that an operator is self-adjoint or unitary, we
always mean that it is self-adjoint or unitary with respect to the scalar
product in $\cH$. 

We denote by $\cB(\cH^{s_1};\cH^{s_2})$ the space of bounded linear
operators from $\cH^{s_1}$ to $\cH^{s_2}$. {As usual in the
  framework of pseudodifferential calculus, we will identify two 
  operators $A\in\cB(\cH^{s_1};\cH^{s_2})$ and
  $A'\in\cB(\cH^{s'_1};\cH^{s'_2})$ with $s_1'>s_1$ if
  $A\big|_{\cH^{s_1'}}=A'$. Correspondingly we will write
  $A\in\cB(\cH^{s_1};\cH^{s_2})\cap \cB(\cH^{s'_1};\cH^{s'_2})$}

We will also
consider the space $\bigcap_{s\in\R}\cB(\cH^s,\cH^{s-m})$ which is a
Fr\'echet space when endowed with the semi-norms $
\Vert \cdot \Vert_{\cB(\cH^s,\cH^{s-m})}$.

\begin{definition}
	\label{smoothing}
We will say that $F$ is of {\em order $m$} if
$F \in \bigcap_{s\in\R}\cB(\cH^s,\cH^{s-m})$. In the case of negative
$m=-N$, 
we will also say that $F$ is {\em $N$-smoothing}.
\end{definition}
\begin{definition}\label{def.smooth}
	If $\cF$ is a Fr\'echet (or a Banach) space, we denote
	by $C^k_b\left(\R^d; \cF\right)$ the space of the functions
	$F\in C^{k}\left(\R^d; \cF\right)$,  such that all the
	seminorms of $d^jF$ are bounded uniformly over $\R^d$ for all $j\leq k$. If this is true for all $k$ we write
	$F\in C^\infty_b\left(\R^d; \cF\right)$ .
\end{definition}

{Given $0 \leq \rho_0< 1$,} we introduce a family of 
algebras $\cA_\rho$ with $\rho\in(\rho_0,1]$ of operators encoding the
fundamental properties of pseudodifferential operators. 

For $m\in\R$ and $\rho\in({\rho_0},1]$ let $\rA{\rho}^m$ be a linear subspace of $
\bigcap_{s\in\R}\cB(\cH^s,\cH^{s-m})$ and define
$\cA_\rho:=\bigcup_{m\in\R}\rA{\rho}^m$.  

\vskip5pt
\noindent
\begin{assumption}{I}\label{assumption.I}
	\textbf{}
	\begin{itemize}
		\item[i.]	For each $m\in \R$, one has $K_0^m\in\rA{1}^m$.
		\item[ii.]
		For each $m\in\R$ and $\rho \in ({\rho_0}, 1]$,
		$\rA{\rho}^m$ is a Fr\'echet space for a family of semi-norms
		$\{ \wp^{m}_{\rho,j} \}_{j\geq 1}$  such that the embedding
		$\rA{\rho}^m\hookrightarrow  \bigcap_{s \in \R}
		\cB(\cH^s,\cH^{s-m})$ is continuous.
		\\
		If $m^\prime {<} m$ then $\rA{\rho}^{m^\prime}{\hookrightarrow }\rA{\rho}^m$ with a continuous embedding.
		\\
		If $\rho_1< \rho_2$, then
		$\rA{\rho_2}^m\hookrightarrow \rA{\rho_1}^m$ with
		continuous embedding.
		\item[iii.] $\forall m,n\in \R$: if $F\in \rA{\rho}^m$ and
		$G\in\rA{\rho}^n$ then $FG\in\rA{\rho}^{m+n}$ and the map
		$(F,G)\mapsto FG$ is continuous from
		$\rA{\rho}^{m}\times\rA{\rho}^{n}$ into $\rA{\rho}^{m+n}$. 
		\item[iv.]  If
		$F\in \rA{\rho}^m$ and $G\in\rA{\rho}^n$ then the commutator
		$[F,G]\in\rA{\rho}^{m+n-\rho}$ and the map $(F,G)\mapsto [F,G]$ is
		continuous from $\rA{\rho}^{m}\times\rA{\rho}^{n}$
		into
		$\rA{\rho}^{m+n-\rho}$.
		\item[v.] 	$\cA_\rho$ is closed under perturbations by smoothing
		operators: let $F:
		\cH^{+\infty}\rightarrow\cH^{-\infty}$ be a linear map. If there exists
		$m\in\R$ such that for every $N>0$ we have a decomposition $F=
		F^{(N)}+ S^{(N)}$, with $F^{(N)} \in\rA{\rho}^m$ and $S^{(N)}$
		is $N$-smoothing, then $F\in\rA{\rho}^m$.
		\item[vi.] 	If $F \in \rA{\rho}^m$, then also the adjoint operator
		$F^* \in \rA{\rho}^m$. 
		\item[vii.] 	For any $F\in\rA{1}^m$, any $G\in\rA{1}^1$, the map
		$\R\ni t\mapsto {\rm e}^{\im tG}\, F \, {\rm
			e}^{-\im tG}\in C^0_b(\R, \rA{1}^m)$.
	\end{itemize}
\end{assumption}

\begin{remark}
	\label{estendo}
	Property {I.iv} is the one which makes the algebras $\cA_\rho$
        different for different $\rho$.
        \\
          {Property I.vii is an abstract version of the Egorov theorem.} 
\end{remark}

\subsection{{Globally integrable quantum systems}}

The idea is to define a quantum integrable system as a system whose
Hamiltonian operator can be written as a function of some \emph{action
  operators}.  The action operators are $d$ self-adjoint pairwise
  commuting operators (in the sense that their projection-valued measures commute) $A_1,...,A_{d }$, 
   fulfilling the
  following Assumption \ref{assumption.A}. 

\begin{definition}
	\label{horma} Let $m \in \R$ and let $0<\varsigma\leq 1$ be a parameter;
	a function $f\in C^\infty(\R^d)$ is said to be a symbol of class
	$S^m_\varsigma$ if $\forall \alpha \in \N^d$ one has
	\begin{equation}
	\label{hor}
{\sup_{  a \in \R^d } \frac{\left|\partial^\alpha f(a)\right|}{\langle a\rangle^{m-\varsigma|\alpha|}}	<\infty}
	\end{equation}
	where $\langle a\rangle:=\sqrt{1+\sum_ja_j^2}$.
\end{definition}}
{The quantities at l.h.s. of \eqref{hor}  form a family of seminorms for $S^m_{\varsigma}$.}
  
\vskip 5pt
\noindent
\begin{assumption}{A}\label{assumption.A}
	\textbf{}
	\begin{itemize}
		\item[i.] For $j=1,...,d$, the operators $A_j$ fulfill $A_j\in\rA1^1$.  
		\item[ii.] $\exists c_1>0$ s.t.\ $c_1 K_0^2<
		\uno+\sum_{j=0}^d A_j^2$.
		\item[iii.] There exist a convex closed cone $\cC\subseteq \R^{d }$
		and a vector
		$
		\kappa=(\kappa_1,...,\kappa_{d })\in\R^d
		$,
		such that the joint spectrum $\Lambda$ of the $A_j$'s
		fulfills
		\begin{equation}
		\label{I1}
		\Lambda\subset (\Z^{d}+\kappa)\cap \cC\ .
		\end{equation}
		\item[iv.]
		{
			There exist $\varsigma_0 \in [0, 1)$ and an increasing continuous function
			$(\varsigma_0, 1]\ni \varsigma\mapsto \rho(\varsigma)\in (\rho_0,1]$,
			with $\rho(1)=1$, 
			s.t., if $f\in S^m_\varsigma$, then
			$
			f(A_1,...,A_d)\in \rA{\rho(\varsigma)}^{m}
			$.
		}
		Furthermore its seminorms depend only on the seminorms of $f$
		and on the seminorms of the $A_j$'s. 
	\end{itemize}
\end{assumption}
\begin{remark}
		\label{2.1}
	{	By {A.ii}, the operator $\langle A \rangle^2:=\uno+\sum_{j=0}^d A_j^2$ has compact inverse,
	  therefore  {it has pure point spectrum and there exists a
          basis of $\cH$ formed by eigenfunctions of $\langle
          A\rangle^2$. Since the operators $A_j$ {pairwise commute 
          and commute} with $\langle A\rangle^2$, there exists also a basis
          $\left\{\psi_{L} \right\}$ of
          $\cH$ formed by eigenfunctions common to all these operators. }
}
 {Then the joint spectrum $\Lambda$ of the operators $A_j$ is
defined as the set of the ${a = (a_1, \dots, a_d)}\in\R^d$ s.t.\ there
exists $L$ with 
	\begin{equation}
	\label{lambda}
        A_j\psi_L=a_j\psi_L
       \ ,\quad \forall j=1,...,d\ .
	\end{equation}
}
\end{remark}

	We are now ready to give our definition of a globally integrable quantum system:
	\begin{definition}[Globally integrable quantum system] \label{def.globally.integrally.quantically}
		 We say that $H_0$ is a \emph{globally integrable quantum system} if there exists $h_0 \in C^\infty(\R^d;\R)$ such that
\begin{equation}\label{giq}
	H_0 = h_0(A_1, \dots, A_d)\,,
\end{equation}
where $A_1, \dots, A_d$ satisfy Assumption \ref{assumption.A} and the function
\eqref{giq} is spectrally defined. 
        \end{definition}

\subsection{The statement}

{To state our assumptions on the function $h_0$ defining $H_0$ we still
  need a couple of definitions.} First we recall that a function  $f\in C^{\infty}(\R^d\setminus\left\{0\right\})$ is
	said to be homogeneous of degree $m$ if  $	f(\lambda
        a)=\lambda^mf(a)$, $\forall \lambda>0$.

Homogeneous functions are typically singular at the origin, but, in
the context of pseudodifferential operators, the behavior of
functions in a neighborhood of the origin is not important. This is
captured by the next definition. 

        \begin{definition}
	A function $f\in C^{\infty}(\R^d)$ will be said to be homogeneous of
	degree $m$ at infinity if it fulfills $f(\lambda a)=\lambda ^m
        f(a)$, $\forall \lambda{>1}$ and $\forall a\in \R^d\setminus
        B_{1/4}$, where $B_r$ is the ball
	of radius $r$ centered at the origin.
        \end{definition}

     We recall, from \cite{GCB}, the definition of steepness:
\begin{definition}[Steepness]	\label{steep.def} 
	Let $\cU\subset \R^d$ be a bounded connected open set with nonempty
	interior.  A function $h_0 \in C^ 1 (\cU )$, is said to be steep in
	$\cU$ with {steepness radius $\tr$,} steepness indices $\alpha_1{\geq1},\dots
	,\alpha_{d-1}{\geq1}$ and (strictly positive) steepness coefficients $\tB_ 1
	, . . . , \tB_{d-1}$, if its gradient $\omega_i(a):=\frac{\partial
		h_0}{\partial a_i}(a)$ fulfills:
	$\displaystyle{\inf_{a\in\cU}\norma{\omega(a)}>0}$ and for any
	$a\in\cU$ and for any $s$ dimensional linear subspace $M\subset \R^d$
	orthogonal to $\omega(a)$, one has
	\begin{equation}
	\label{steep}
	\max_{0\leq\eta\leq\xi}\min_{u\in M:\norm u=1}\norma{\Pi_M
		\omega(a+\eta u)}\geq \tB_s\xi^{\alpha_s} \quad \forall \xi \in (0, \tr]\ ,
	\end{equation}
	where $\Pi_M$ is the orthogonal projector on $M$; the quantities $u$
	and $\eta$ are also subject to the limitation $a+\eta u\in\cU$. 
\end{definition}

\begin{remark}
  \label{steep.nider}
It is well known that steepness is generic. 
Examples of steep functions are given by functions which are convex or quasiconvex. In
the applications we will verify steepness by verifying an equivalent
condition due to Niederman \cite{Nied06} (see Theorem \ref{nie06} below).
\end{remark}

On $H_0$ we assume:
%

\begin{assumption}{H}\label{assumption.H}
	\textbf{}
	\begin{itemize}
		\item[i.] {$H_0$ is the Hamiltonian of a globally integrable quantum system and the function $h_0$ of \eqref{giq} is homogeneous of
			degree $\td>1$ at infinity.}
		\item[ii.] There exists an open set $\cU\subset\R^d$, s.t.\ $\cU \supset \overline{(B_{2}\setminus
			B_{1/2})\cap \cC}$, with the property that $h_0$
		{is} steep on $\cU$.
	\end{itemize}
\end{assumption}

\begin{theorem}
	\label{main}
	Let $H=H(t)$ be of the form
	\begin{equation}
	\label{H.mai}
	H(t):=H_0+V(t)
	\end{equation}
	with {$H_0$ the Hamiltonian of a globally integrable quantum
          system}. Assume that Assumption \ref{assumption.H} holds and that $V(\cdot) \in
	C^\infty_b\left(\R; \rA{1}^\tb \right)$ is a family of self-adjoint
	operators. Assume $\tb< \td$; {then for any $s\geq 0$ and for any initial datum $\psi \in \cH^s$ there exists a unique global solution $\psi(t) := \cU(t, \tau)\psi \in \cH^s$ of the initial value problem
		\begin{equation}
		\label{p.abs}
		\im \partial_t \psi(t) = H(t) \psi(t)\, ,  \quad \psi(\tau) = \psi\,,
		\end{equation}
	}
	furthermore, for any $s>0$ and $\varep>0$ there exists a positive constant $K_{s, \varep}$ such that for any $\psi \in \cH^s$
	\begin{equation}\label{growth.eq}
	\|\cU(t, \tau) \psi\|_s \leq K_{s, \varep} \langle t -\tau
        \rangle^{\varep}\| \psi\|_s\ , \quad \forall t, \tau \in \R\,.
	\end{equation}
\end{theorem}

\section{Applications}
\subsection{Application 1: the anharmonic oscillator in dimension 2}\label{anha}

We define $\cH:=L^2(\R^2)$ and, for $\ell\in \N$, $\ell\geq 2$, consider the Hamiltonian of  the quantum anharmonic oscillator
\begin{equation}
\label{anha.H0}
H_0:=-\frac{\Delta}{2}+\frac{\|x\|^{2\ell}}{2\ell}\ ,\quad x\in\R^2\ .
\end{equation}
In order to define the scale of Hilbert spaces $\cH^s$, we define
$
K_0:=(\uno +H_0)^{\frac{\ell+1}{2\ell}}
$, whose principal symbol is
\begin{equation}
  \label{k0.ana}
{\tt k}_0(x,\xi):=\Big(1+\frac{\|x\|^{2\ell}}{2\ell}+\frac{\|\xi\|^2}{2}\Big)^{\frac{\ell+1}{2\ell}}\ .
\end{equation}
For $\rho\in\left(\frac{\ell-1}{\ell+1},1\right]$, define
\begin{equation}
\label{le delte}
\delta_1:=\frac{1}{2}\left(\rho-\frac{\ell-1}{\ell+1}\right)\ ,\quad
\delta_2:=\frac{1}{2}\left(\rho+\frac{\ell-1}{\ell+1}\right)  \ .
\end{equation}

\begin{definition}
	\label{symbol.a1}
	Given $f \in C^\infty(\R^4)$, we will write $f \in S^{m}_{AN,\rho}$ if $\forall \alpha, \beta \in \N^2$, there exists $C_{\alpha, \beta} >0$
	s.t.
	\begin{equation}
	\label{es.7}
	\vert \partial_x^\alpha \, \partial_\xi^\beta f(
	x,\xi)\vert \leq C_{\alpha,\beta}
	\ \left(\tk_0(x,\xi)\right)^{m-\delta_1{|\alpha|}-\delta_2
		{|\beta}|} \quad \forall (x, \xi) \in \R^4  \ ,
	\end{equation}
	with $\delta_1, \delta_2$ given by \eqref{le delte}.
	We will say that an operator $F$ is a
	pseudodifferential operator of class $\cA^{m}_\rho$
	if there exists a symbol $f \in S^{m}_{AN,\rho}$
	s.t.\ $F$ is the Weyl
	quantization of $f$.
\end{definition}

\begin{remark}
	\label{hr}
	When $\rho=1$ the class of symbol $\cA^{m}_1$ reduces to the standard classes
	used to study the anharmonic oscillator (see
	e.g. \cite{HR82}). The case with $\rho<1$ was studied in \cite{anarmonico}.
\end{remark}

The properties I are immediate consequences of standard
pseudodifferential calculus in $\R^{4}$. {Following \cite{CdV2,
		charbonnel, anarmonico}, the operators $A_1, A_2$ will be constructed in
	Subsection \ref{oscilla} by quantizing the classical
        actions. {Assumptions \ref{assumption.A} and \ref{assumption.H}} will be verified in Subsection
        \ref{oscilla}, so that we have the following:

\begin{theorem}
	\label{anhar}
	Consider the Schr\"odinger equation \eqref{p.abs} with $H_0$ given by
	\eqref{anha.H0} and $V(.)\in C^\infty_b(\R;\cA^{\tb}_1)$, with
	$\tb<\frac{2\ell}{\ell+1}$, then the corresponding evolution
	operator fulfills \eqref{growth.eq}. 
\end{theorem}

\subsection{Application 2: manifolds with  {globally} integrable geodesic flow}\label{appli}

Let $(M,g)$ be a compact $n$-dimensional Riemannian manifold without
boundary; to fit our scheme we define $\cH:=L^2(M)$,
$K_0:=\sqrt{\uno-\Delta_g,}$ with $\Delta_g$ the negative Laplace-Beltrami operator relative to the metric $g$, so that $\cH^s$
coincides with the classical Sobolev space $H^s$.  {In this case, the
  pseudodifferential operators are the standard ones defined by
  H\"ormander. Precisely we give the following definition.}

\begin{definition}
	\label{psM}
	A function $f\in C^{\infty}(T^*M)$ is said to
	be a symbol of class $S^m_{H,{\varrho}}$, if, when written in any canonical
	coordinate system (in the sense of $T^*M$) it fulfills
	\begin{equation}
	\label{psM.1}
	\left|\partial^\alpha_x\partial^\beta_\xi f(x,\xi)\right|\leq
	C_{\alpha,\beta}\langle \xi\rangle^{m-{\varrho}|\beta|+(1-{\varrho})|\alpha|}\ ,\quad \forall
	\alpha,\beta\in\N^n\,, \quad \forall (x, \csi) \in T^*M\ .
	\end{equation}
\end{definition}

\begin{definition}
	\label{psM.10}
	We say that $F\in\cA^m_\rho$, if it is a pseudodifferential operator
	(in the sense of H\"ormander \cite{ho}) with Weyl
	symbol of class $S^m_{H, {\varrho}}$, {with $\displaystyle{\varrho = \dfrac{\rho + 1}{2}}$}. 
\end{definition}

Then Assumption \ref{assumption.I} holds. {In particular, the commutator between two pseudodifferential operators gains $\rho$ with respect to their product.} {Furthermore, Assumption \ref{assumption.A}.iv} with $\rho(\varsigma) =
	2 \varsigma - 1$ follows from functional calculus.
{We are now going to study some specific manifolds $M$. The applications are dealt with in
  different subsections, since the construction is different in each
  specific case.}
{The result will always be that the solution of the Schr\"odinger
  equation
\begin{equation}
\label{manifold}
\im\frac{\partial \psi}{\partial t}=(-\Delta_g+V(t))\psi\ ,\quad
\psi\in H^s(M)
\end{equation}
with $V(.)\in C^{\infty}_b(\R;\cA_1^\tb) $, $\tb<2$, fulfills
the estimate \eqref{growth.eq}.}

%

\subsubsection{Flat tori}\label{s.tori}

Let ${\bf e}_1, {\bf e}_2, \ldots, {\bf e}_n$ be a basis of $\R^n$
and let $\Gamma:\textrm{span}_{\Z}\left\{{\bf e}_1, {\bf e}_2, \ldots, {\bf
  e}_n \right\}$, be the maximal lattice that they generate. Define $
M\equiv \T_\Gamma := \R^n / \Gamma\,.$ By introducing in $\T_\Gamma$
the basis of the vectors ${\bf e}_i$, the Laplacian is transformed in
the operator $H_0:=\sum_{k, l}g^{kl}(-\im \partial_k)(-\im \partial_l)
$, with $g^{kl}$ the inverse matrix of $g_{jk}:={\bf e}_j\cdot{\bf
  e}_k$. In this case one has $A_j:=-\im \partial_j$, $j=1,...,n=
d$ and $h_0(\csi):=\sum_{kl}g^{kl}\xi_l\xi_k,$ which is convex and
thus steep. So Theorem \ref{main} applies and we get the estimate
\eqref{growth.eq}. This result was already obtained in
\cite{BLMgrowth}, which improved the results
\cite{Bourgain,delort, berti_maspero}.

\subsubsection{Zoll manifolds}\label{zoll}

We recall that a Zoll manifold is a compact manifold s.t.\ all its
geodesics are closed; the typical example of a Zoll manifold is a sphere. By Theorem 1 of \cite{cdv}, there exists
a pseudodifferential operator $Q$ of order  -1, commuting with $-\Delta_g$,
s.t.\ spec$(\sqrt{-\Delta_g}+Q)\subset \N+\kappa$, with $\kappa\geq
0$. We put $A:=\sqrt{-\Delta_g}+Q$ {and $h_0(a):= a^2$}. {We remark
that in this case one has $d=1$. We thus get that the solutions of \eqref{manifold} on a Zoll manifold
fulfill \eqref{growth.eq}.} We recall that this result was already
obtained in \cite{BGMR}.

\subsubsection{Rotation-invariant surfaces}\label{rotazioni}


Consider a real {\it analytic} function $f:\R^3\to\R$
invariant by rotations around the $z$ axis, and assume it is a
submersion at $f(x,y,z)=1$. Denote by $M$ the level surface
$f(x,y,z)=1$, {suppose that $M$ is diffeomorphic to $\mathbb{S}^2$} and endow it by the natural metric $g$ induced by the
euclidean metric of $\R^3$, then $M$ has integrable geodesic flow. 
{Following \cite{CdV2}, we introduce suitable coordinates in $M$ as
follows: let $N$ and $S$ be the
north and the south poles (intersection of $M$ with the $z$ axis) and denote by
$\theta\in[0,L]$ the curvilinear abscissa along the geodesic given by
the intersection of $M$ with the $xz$ plane; we orient it as going
from $N$ to $S$ and consider also the cylindrical coordinates
$(r,\phi,z)$ of $\R^3$: we will use the coordinates}
$
(\theta,\phi)\in (0,L)\times(0,2\pi)
$
as coordinates in $M$. {Using such coordinates, one can write the
equation of $M$ {by} expressing the cylindrical coordinates of a point in
$\R^3$ as a function of $(\theta,\phi)$ getting
$$
M=\left\{ (r(\theta),\phi,z(\theta))\ ,\quad (\theta,\phi)\in
(0,L)\times\T^1\right\} \,.
$$} Since $\theta$ is a geodesic parameter, the metric takes the form
$ g=r^2(\theta)d\phi^2+d\theta^2 $.
We assume that the function $r(\theta)$ has only one
critical point $\theta_0\in(0,L)$. Furthermore,
we need to ensure 
steepness. To this aim, consider the following
Taylor expansion at $\theta=\theta_0$: 
\begin{equation}
\label{taylor.ex}
\frac{1}{2r^2(\theta)}=\beta_0+\frac{1}{2}\beta_2(\theta-\theta_0)^2+\frac{1}{3!}\beta_3(\theta-\theta_0)^3+\frac{1}{4!}\beta_4(\theta-\theta_0)^4+O(|\theta-\theta_0|^5)\ .
\end{equation}

\begin{theorem}
	\label{rotation}
	Consider the Schr\"odinger equation \eqref{manifold} with $V(.)\in C^\infty_b(\R;\cA^{\tb}_1)$, with
	$\tb<2$. Assume also that $\beta_2\not=0$, and that 
	\begin{equation}
	\label{steep.riv.2}
	\beta_0\frac{-5\beta_3^2+3\beta_2\beta_4}{24\beta_2^2}-\beta_2\not=0 \,,
	\end{equation}
	then Assumption \ref{assumption.H} holds.
\end{theorem}

This theorem will be proved in Subsection \ref{ruota}.
{In this example, the actions were actually constructed in \cite{CdV2} by quantizing the classical action variables.}
\begin{remark}
  \label{delort_ruota}
For the case of \emph{bounded potentials}, this theorem was proved
in \cite{delort} where the condition \eqref{steep.riv.2} was not
required.
\end{remark}

\subsubsection{Compact, simply connected Lie groups} \label{liegroups}

Let $M\equiv G$ be a simply connected compact Lie group endowed with the
bi-invariant metric $g$. To apply Theorem \ref{main} to equation
\eqref{manifold} on $G$ we use the intrinsic formulation of
pseudodifferential calculus on Lie groups, developed in
\cite{ruzhansky_turunen,Fischer}. In particular this will be needed to
construct the quantum actions $A_j$ and to verify their
properties. We
remark that Lie groups which are simply connected and
compact are given by  $\textrm{SU}(n)$ with
$n\geq2$, $\textrm{Sp}(n)$ with $n \geq
3$, $\textrm{Spin}(n)$ with $n \geq 7$, and  $G_2$, $F_4$, $E_6$, $E_7$, $E_8$ {and their direct products. Our result also extends to direct products of compact, simply connected Lie groups with tori of any dimension}.}
An extension to more general compact Lie groups and homogeneous spaces
can also be obtained; the details are left for a future work. 

The starting point of the construction is the fact that in the
intrinsic Fourier calculus in Lie groups,
the Fourier coefficients of a smooth function are labeled by the
irreducible unitary representations of the group and each Fourier
coefficient is a unitary operator in the representation
space. 

More precisely, denote by $\widehat{G}$ the set of unitary irreducible
representations of $G$ modulo unitary equivalence and by $\Rep(G)$ the
set of unitary representations of $G$ (still modulo equivalence). Given
$\xi\in \Rep(G)$, denote by $\cH_\xi$ the corresponding
representation space; then the Fourier coefficients of a function
$\psi:G\to\R $ are a sequence $\{\hat \psi_\xi\ |\ \xi\in\widehat G \}$
with $\hat\psi_\xi\in\cB(\cH_\xi)$.

There is a way of defining symbols
of pseudodifferential operators as maps $\sigma$,
\begin{equation}
\label{lie.symb}
G\times \Rep(G)\ni (x,\xi)\mapsto \sigma(x,\xi)\in\cB(\cH_\xi)\ , 
\end{equation}
with suitable properties {(see Definition \ref{simbolo.fischer} below for the
precise definition taken from \cite{Fischer})}. Actually a symbol is usually defined by its
action on $\widehat G$ and extended to $\Rep(G)$ by direct sum. 

To define the actions we need a further step in the theory of Lie
groups: to a representation $\xi\in\widehat G$, one associates its
{highest} weight $\tw_\xi$,  and it turns out (see
e.g. \cite{fulton_harris}) that there is a 1-1 correspondence
between the elements of $\widehat G$ and the elements of the cone
$\Lambda^+(G)$ of dominant weights, defined by
\begin{equation}\label{cono.pesi}
\Lambda^+(G) = \left\{ \tw \in \R^d\ \left|\ \tw = \tw^1 \tf_1 + \dots
+ \tw^d \tf_d\,, \quad \tw^j \in \N\ , \  \forall j= 1, \dots, d \right.\right\}\,,
\end{equation}
where $\tf_1, \dots, \tf_d \in \R^d$ are the \emph{fundamental weights} of
$G$. In the following we also denote $\underline \tf:=\sum_{j=1}^d
\tf_j\in \R^d$ and, given a
dominant weight $\tw$, we denote by $\tw^j$ its components on the basis
$\tf_j$, namely the numbers such that
$
\tw=\sum_{j=1}^d\tw^j \tf_j
$.
With this notation, the Laplacian $-\Delta_g$ acts in Fourier space
as follows:
\begin{equation}
\label{lie.lapa}
\widehat{(-\Delta_g \phi)}_\xi=
(\left\|\tw_\xi+ \underline \tf \right\|^2-\|\underline \tf\|^2)\hat \phi_\xi\ ,
\end{equation}
and its symbol is given by
$$
\sigma_{-\Delta_g}(\xi)=   (\left\|\tw_\xi+\underline
\tf\right\|^2-\|\underline \tf\|^2)\uno_{\cH_\xi} \ .
$$

We are now ready to define the quantum actions $A_1, \dots, A_d$ as
the operators acting in Fourier space as follows:
\begin{equation}
\label{lie.action}
\widehat{(A_j \phi)}_\xi:=
(\tw_\xi^j+1)\hat \phi_\xi\ ,
\end{equation}
whose symbol is given by
\begin{equation}\label{def.sigmap}
\sigma_{A_j}(\xi)= \left(\tw_\xi^j+1 \right)\uno_{\cH_\xi} \ .
\end{equation}
By direct computation one can see that
the operators $A_j$ commute, that 
their joint spectrum is
$
\Lambda=\N^d+\kappa\ ,\quad \kappa=(1,...,1)$,
and that
\begin{equation}\label{lie.lapla.azio}
-\Delta_g=\sum_{i, j = 1}^{d} A_i A_j\, \tf_i \cdot \tf_j - \|\underline
\tf\|^2\ , 
\end{equation}
so that we can define 
\begin{equation}
\label{lie.hzero}
h_0(A):= \sum_{i, j = 1}^{n} A_i A_j\, \tf_i \cdot \tf_j\ .
\end{equation}
 Note that $h_0$ is homogeneous of degree 2, convex and thus steep.
In Section \ref{sec.lie.proofs} we will prove that the $A_j$'s are
pseudodifferential operators, so that we can apply Theorem \ref{main}
and deduce that the estimate \eqref{growth.eq}
holds for the solutions of the equation \eqref{manifold} on a compact,
simply connected Lie group.

\vskip10 pt

\part{Proofs}

\section{Analytic part}\label{NF.1}

We start by fixing some notations and definitions that will be used in
the rest of the paper,  {then we will state and prove the normal
  form Lemma.}

Given two real valued functions $f$ and $g$,
sometimes we will use the notation $f \lesssim g$ to mean that there
exists a constant $C>0$, independent of all the relevant quantities,
such that $f \leq C g$. If $f \lesssim g$ and $g \lesssim f$, we will write $f
\simeq g$.

{We recall that we denote
  $$
\omega(a):=\frac{\partial h_0}{\partial a}(a)\ ,
$$
which is homogeneous at infinity of degree}
\begin{equation}
\label{M}
{\mom:=\td-1}\, .
\end{equation}
Furthermore, given $\delta$ such that $\max\{0,\, {\varsigma_0 + }
\mom-1\}< \delta <\mom,$  {where $\varsigma_0$ is the quantity defined
  in Assumption \ref{assumption.A}.{iv}}, we set
\begin{equation}\label{varsigma}
\varsigma := 1 - (\mom -\delta)\,\ ,\quad
\quad \rho:=\rho(\varsigma)\,,
\end{equation}
where $\rho(\cdot)$ is the function defined in Assumption \ref{assumption.A}.{iv}.

	\begin{definition}
		\label{eigenfunctions}
		Given a joint eigenvalue $\ta = (\ta_1, \dots, \ta_d) \in\Lambda$ of the $A_j$'s, we consider the
		corresponding joint eigenspace, namely the space
		$\Sigma_{\ta}\subset \operatorname{dom}(K_0)$
		with the property that
		\begin{equation}
		\label{jointe}
		\psi\in\Sigma_{\ta} \quad \iff \quad A_j\psi=\ta_j\psi\ ,\quad\forall
		j=1,...,d\ .
		\end{equation}
		The orthogonal projector on $\Sigma_{\ta}$ will be denoted by $\Pi_{\ta}$.
	\end{definition}
	\begin{remark}
		\label{norme}
		Given $\psi\in \cH$, one can consider its spectral decomposition,
		namely
		\begin{equation}
		\label{deco}
		\psi=\sum_{\ta\in\Lambda}\Pi_{\ta}\psi\ ,
		\end{equation}
		then by Assumption \ref{assumption.A}.{ii} $\forall s \in \N$ there exist $c_{1,s}, c_{2,s} >0$ such that one has that
		\begin{equation}
		\label{N.2}
		c_{1,s}\left\|\psi\right\|_s^2\leq \sum_{\ta\in\Lambda}\langle
		\ta\rangle^{2s} \norma{\Pi_\ta\psi}^2_0\leq
		c_{2,s}\left\|\psi\right\|_s^2\,.
		\end{equation}
	\end{remark}

        Given {$\mu \in (0, 1)$ and $\tR>1$ (typically $\mu\ll 1$ and $\tR\gg 1$) },
we give the following definitions:
\begin{definition}
	\label{res}
	We say that a point $a\in \Lambda$ is resonant with
	$k\in\Z^d\setminus\{0\} $ if {$\|a\| \geq \tR$} and
	\begin{equation}
	\label{reso.1}
	|\omega(a)\cdot k| \leq \|a \|^\delta \|k\| \quad \text{and}\quad \|k\| \leq \|a \|^\mu\,.
	\end{equation}
	If $a \in \Lambda$ does not satisfy \eqref{reso.1}, we say that $a$ is nonresonant with $k$.
\end{definition}
\begin{definition}[Normal form]\label{def.nf}
	We say that an operator $Z \in \rA{\rho}^m$ is \emph{in normal
		form} if
	\begin{equation}
	\label{nor.for.def}
 {\exists \psi\in\cH\  s.t.\ }	\langle\Pi_a\psi;Z\Pi_b\psi\rangle\not=0
	\end{equation}
	implies that either $a$ is resonant with $b-a$, or $b$ is
	resonant with $b-a$.
\end{definition}

\begin{definition}
	\label{conju}
	We say that a family of unitary operators $U(t)$, conjugates $H$ to
	$H^+$, if, when $\psi(t)=U(t)\phi(t)$, one has
	\begin{equation}
	\label{conj.1}
	\im \dot \psi(t)=H(t)\psi(t)\quad \iff \quad \im\dot\phi(t)=H^+(t)\phi(t)\ .
	\end{equation}
\end{definition}

We are going to prove the following normal form theorem

\begin{theorem}[Normal form lemma]\label{norm.form}
	Let $H$ be as in equation \eqref{H.mai}, with $V \in
	\timereg\left(\R;\rA1 ^\tb\right)$, $\tb<\td$, and assume that
	$V(t)$ is a family of self-adjoint operators. There exists
        $0<\delta_*<\mom$ such that, if $\delta_*<\delta<\mom$, and 
	\begin{equation}\label{sigma.1}
	\da:=\min\left\{2\rho+\delta-\td;\rho+\delta-\tb;\delta
	\right\}>0\ ,
	\end{equation}  {then} for any $N \in
	\N$, for any $
	\mu, \tR$ {as above} there exists a time-dependent family of unitary
	maps $U_N(t)$ which conjugates $H$ to 
	\begin{equation}\label{eq.in.forma.1}
	H^{(N)}:=	H_0+Z_N(t) + R^{(N)}(t)\,,
	\end{equation}
	and the following properties hold
	\begin{enumerate}
		\item $Z_N \in \timereg\left(\R; \rA{\rho}^\tb\right)$ is a
		family of self-adjoint operators in normal form;
		\item $R^{(N)} \in \timereg\left(\R;\rA{\rho}^{\tb-\da
			N}\right)$, is a
		family of self-adjoint operators;
		\item For any $s \geq 0,$ $U_N,\ U_N^{-1} \in {L^\infty} \left( \R; \cB(\cH^s;\cH^s)\right)\,.$ 
	\end{enumerate}
\end{theorem}
The rest of this section is devoted to the proof of this
theorem. Actually this is the generalization to the abstract setting
of  theorems proven in \cite{BLMres, anarmonico} so
we only present in detail the points of the proofs different from those of
\cite{BLMres,anarmonico}.

The conjugating maps $U_N(t)$ that one looks for are compositions of
	maps of the form $e^{-\im G(t)}$, with $G(.)\in
	C^{\infty}_b(\R;\rA{\rho} ^\eta)$ a family of
	self-adjoint pseudodifferential operators with $\eta<\rho$ {and for each fixed $t$, $e^{-\im G(t)}$ is the complex exponential of the operator $-\im G(t)$, as defined through functional calculus for self-adjoint operators.}
    The detailed study of the
	properties of $e^{-\im G(t)}$ {was done in a
	context very similar to the present one in \cite{BGMR}, to which we refer for more details}.

\begin{lemma}
	\label{hpiu}
	Let $G\in\cinfinito\eta$, $\eta<\rho$ and $H\in\cinfinito m$, be
	families of self-adjoint operators; then $e^{-\im G(t)}$ conjugates $H$
	to $H^+$ given by
	\begin{equation}
	\label{Hpiu1}
	H^+=H-{\im
          [H,G]+\frac{1}{2}[[H;G];G]+\cinfinito{m+3(\eta-\rho)}}+ 
	\cinfinito{\eta}\ .
	\end{equation}
\end{lemma}
\begin{proof}
 {	
 By Lemma 3.1 of \cite{BGMR}, one obtains
 \begin{equation}\label{full.detail}
  H^+ = e^{\im G(t)} H e^{-\im G(t)} - \int_{0}^1 e^{\im s G(t)} \partial_t G e^{-\im s G(t)}\ d s\,.
 \end{equation}
 Then one applies Lemma 3.2 of \cite{BGMR} with $M=2$ to the first summand of \eqref{full.detail} and Lemma 3.2 of \cite{BGMR} with $M=1$ to the second summand of \eqref{full.detail}.
}
\end{proof}
We use this formula to compute the structure of the transformed
Hamiltonian. To this end remark that, since {$H_0\in\rA{1}^\td$ (by
	A.iv and H.i) and $V\in
	C^{\infty}_b(\R;\cA^\tb_1)$}, taking $G\in\cinfinito\eta$
{with $0<\rho-\eta< (\td - \tb)/2$ and $\eta < \tb$ (which will be our case),} one gets that {$H^+$ has the structure}
\begin{equation}
\label{Hpiu2}
H^+=H_0-\im [H_0,G]+V+\cinfinito{{\td+2(\eta-\rho)}}+
\cinfinito{\eta}\ ,
\end{equation}
{with $\max\{\td + 2 (\eta - \rho),\ \eta\} < \tb$.} Then the idea is to determine $G$ which solves the so called
co-homological equation, namely
\begin{equation}
\label{cohomo}
-\im [H_0,G]+V=Z+lower\ order\ terms
\end{equation}
with $Z$ in normal form and then to iterate the construction. The
solution of \eqref{cohomo} is the
main issue of this section and will be done by developing an abstract
version of the normal form theory of
\cite{BLMnr,BLMres,anarmonico}. This requires some work that
will be done in the next subsections.

\subsection{The Fourier expansion}\label{fourier}

{Here we extend the theory of Sect. 3.3 of \cite{BGMR} (see also
\cite{Bam96,nek_noi}) to the case where $H_0$ is a globally integrable quantum system. The idea is that the conjugation of operators with the unitary groups generated by the quantum actions defines a group action of the torus $\T^d$ on the space of pseudo-differential operators. Such a group action is used to define a Fourier expansion of pseudodifferential operators and to develop in a corresponding way normal form theory. A delicate point consists in solving the co-homological equation (Eq. \eqref{cohomo}), and this is done in Subsection \ref{cutoffs}. We start by giving the following definition.}

\begin{definition}
	\label{fou.1}
	Let $F\in \rA{\rho}^m$ with $\rho \in ({\rho_0}, 1]$, then, {for $k\in\Z^d$,} we define its $k$-th
	Fourier coefficient to be: 
	\begin{equation}\label{fou.2}
	\hat
		F_k:=\frac{1}{(2\pi)^d}\int_{\T^d} e^{\im \vf\cdot A}Fe^{-\im
			\vf\cdot A}
		e^{-\ii
			k\cdot \vf}d\vf\,.
	\end{equation}
\end{definition}
In the following we will use the notation
\begin{equation}
  \label{flus}
F(\vf):= e^{\im \vf\cdot A}Fe^{-\im
	\vf\cdot A}\ .
\end{equation}

	\begin{remark}
		\label{smoo}
		By the formula 
		$$
		\frac{d}{d\vphi_j}\left({\rm e}^{\im \vphi \cdot A}\, F \, {\rm e}^{-\im \vphi\cdot
			A}\right)= {\rm e}^{\im \vphi\cdot A}\, \left(-\im[F;A_j]\right) \, {\rm e}^{-\im \vphi\cdot
			A}
		$$
		Assumption \ref{assumption.I}.vii implies that,  {for $F\in\rA{1}^m$,
                the right-hand side is in $C^0_b(	{\T^d}, \rA{1}^m)$, so that, by
                iteration,}                the map
		$\vphi\mapsto F(\vphi)={\rm e}^{\im \vphi \cdot A}\, F \, {\rm e}^{-\im
			\vphi\cdot A}\in C^\infty_b({\T^d}, \rA{1}^m)$.
	\end{remark}

        \begin{remark}
          \label{delA}
{If $F$ is selfadjoint then, $\forall k\in\Z^d$, one has $(\hat
F_{k})^*=\hat F_{-k}$.}
          \end{remark}

\begin{lemma}
	\label{foou.11}
	Let $F \in
	\rA{1}^m$, then for any $j$ and $N\in \N$ there exist $C>0$
	and $J$, independent of $F,$ such that
	\begin{equation}\label{easy}
	\wp^ m_{1,j}\left({\hat F_k}\right)\leq C \frac{\wp^ m_{1,J}\left({F}\right)}{\langle k\rangle^{N}} \quad \forall k \in \Z^d .
	\end{equation}
	It follows that the series
	\begin{equation}
	\label{f.some}
	F(\vf)=  \sum_{k\in\Z^d}\hat F_ke^{\ii k\cdot
		\vf}\ 
	\end{equation}
	is convergent.
\end{lemma}
The proof is an immediate consequence of Remark \ref{smoo}.

\begin{definition}\label{def.norma}
	For $m\in\R$, $\rho\in({\rho_0},1]$, the set of the operators $F \in
	\rA{\rho}^m$, 
	s.t.\ {$\forall j$, and $\forall N\in\N$} 
	\begin{gather}
	\label{norma.tutto}
	\norfou{m}j{F}:= \sum_{k \in \Z^d} \langle k
	\rangle^{N} \semi{m}{j}{\hat{F}_k}<\infty\,, 
	\end{gather}
	will be denoted by $\rS{\rho}^ m$. This is a Fr\'echet space with the
	family of seminorms \eqref{norma.tutto}. 
\end{definition}

\begin{remark}
	\label{embed}
	By Lemma \ref{foou.11} one has $\rA1^m\hookrightarrow\rS1^ m $ continuously. 
\end{remark}

By proceeding as in the proof of Lemma 5.16 of \cite{anarmonico} one gets
\begin{lemma}\label{lem.prod.n}
	Let $F \in \SF{m}$ and $G \in \SF{m^\prime}.$ Then $FG \in
	\SF{m+m^\prime}$, $[F;G]\in \SF{m+m'-\rho}$ and $\forall N \in \N$, $\forall j$, $\exists J,C$ s.t.\ one has
	\begin{align}
	\label{}
	\label{fou.prodotto}
	\norfs{FG}{m + m^\prime}{j} \leq C \norfs{F}{m}{J}
	\norfs{G}{m^\prime}{J }\,,
	\\
	\label{fou.pro.1}
	\norfs{[F;G]}{m + m^\prime-\rho}{j} \leq C \norfs{F}{m}{J}
	\norfs{G}{m^\prime}{J }\,.
	\end{align}
\end{lemma}

\begin{remark}
	\label{kj}
	By deriving $F(\vf)$ with respect to $\vf_j$ one gets
	\begin{equation*}
	\frac{\partial F}{\partial\vf_j}(\vf)=\sum_{k}\im k_j\hat F_k e^{\im
		k\cdot \vf}=-\im [F(\vf);A_j]\ , 
	\end{equation*}
	which implies
	\begin{equation}
	\label{der.1}
	\sum_{k}\im k_j\hat F_k =-\im [F;A_j]\ .
	\end{equation}
\end{remark}
That's why this Fourier expansion is useful for the solution of the
cohomological equation.

The main property relating the lattice
$\Lambda$ and the Fourier expansion is given by the following lemma
\begin{lemma}
	\label{lego}
	Let $F\in\rA{\rho}^m$ for some $m$. For any $a, b \in \Lambda$ and
	$k \in \Z^d$, if $\psi_a$ is an eigenfunction corresponding to
	$a$ and $\psi_b$ is an eigenfunction corresponding to $b$, one has
	\begin{equation}
	\label{lego.1}
	\langle\Pi_b\psi;\hat F_k\Pi_a\psi\rangle=\delta_k^{a-b}\langle\Pi_b\psi;F\Pi_a\psi\rangle
	\ .  \end{equation}
\end{lemma}
\begin{proof}
Just compute
\begin{align*}
\langle\Pi_b\psi;\hat
F_k\Pi_a\psi\rangle&=\frac{1}{(2\pi)^d}\int_{\T^d}\langle\Pi_b\psi; e^{\im
	\vf\cdot A}Fe^{-\im
	\vf\cdot A}\Pi_a\psi\rangle e^{-\im \vf\cdot k}d\vf
\\
&=\frac{1}{(2\pi)^d}\int_{\T^d}\langle e^{-\im
	\vf\cdot A}\Pi_b\psi;Fe^{-\im
	\vf\cdot A}\Pi_a\psi\rangle e^{-\im \vf\cdot k}d\vf
\\
&= \frac{1}{(2\pi)^d}\int_{\T^d}\langle e^{-\im
	\vf\cdot b}\Pi_b\psi;Fe^{-\im
	\vf\cdot a}\Pi_a\psi\rangle e^{-\im \vf\cdot k} d\vf
\\
&=\langle\Pi_b\psi;F\Pi_a\psi\rangle \frac{1}{(2\pi)^d}\int_{\T^d} e^{-\im \vf\cdot (k+a-b)}   d\vf\,.
\qedhere
\end{align*} \end{proof}

\subsection{Solution of the cohomological equation} \label{cutoffs} 

In this subsection we are going to prove the following lemma.

\begin{lemma}\label{lem.hom}
	There exists $\delta_*<\mom$ s.t.\ for all $\delta_*<\delta<\mom$, the
	following holds true: define $\varsigma:=\varsigma(\delta)$ and
	$\rho=\rho(\varsigma(\delta))$ according to \eqref{varsigma}, then
	$\forall F\in\rS{\rho}^m$ self-adjoint,
	there exist
	self-adjoint
	operators $G\in\rS\rho ^{m-\delta}$, $Z\in\rS\rho ^{m}$, with $Z$ in
	normal form, s.t.\ 
	\begin{equation}\label{solve.me}
	-\im[H_0; G] + F-Z \in\rS\rho ^{m-(2\rho+\delta-\td)}+\cA_\rho^{-\infty}\ ,
	\end{equation}
	furthermore, for $\delta_*<\delta<\mom$ one has $2\rho+\delta-\td>0$.
\end{lemma}

First, following \cite{anarmonico}, we split the perturbation $F$ in
a resonant, a nonresonant and a smoothing part. This will be done with
the help of suitable pseudodifferential cutoffs.

{Let $\chi\in C^{\infty}(\R,\R)$ be a symmetric cutoff function which
	{is equal to 1 in $[-\frac{1}{2},\frac{1}{2}]$ and has support in $[-1,1]$,}  and
	given $\tR>0$, define $$\chir(t) := \chi(\tR^{-1}
	\|t\|)\ ,\quad t\in\R^d\,,$$
	 {which is of class $C^\infty$ notwithstanding the singularity of $\|t\|$ at $t=0$, since $\chi$ is constantly equal to $1$ in a neighborhood on $0$.}
	\vskip5pt
	With its help we define,
	for $k\in\Z^d\setminus\left\{0\right\}$,
	\begin{align}
	\label{cut-off-piccoli-divisori.1}
	\tilde\chi_k( a) &:= \chi
	\left(\frac{\|k\|}{\|a\|^{\mu}}\right)\,\, ,\quad
	\chi_k( a) := \chi\left({\frac{\omega(a) \cdot
				k}{\| a\|^{ \delta}\|k\|}}  \right)\,,
	\\
	\label{cut-off-piccoli-divisori.4}
	d_k( a) &:= \frac{1}{\omega(a) \cdot
          k}\left(1-\chi\left({\frac{\omega(a) \cdot
          		k}{\| a\|^{ \delta}\|k\|}}\right)\right) \,.
	\end{align}
	We also put                
	\begin{gather} \label{media}
	\chi_0 (a) := 1\,, \quad  {\tilde{\chi}_0 (a):= 1}\,,
	\end{gather}
	$$
	\begin{gathered}
	\chi^T_k(a):=(1-\chir(a))\chi_k(a)\,,\quad
	d^T_k(a):=(1-\chir(a))d_k(a)\,.
	\end{gathered}
	$$
}
By the following lemma the above functions are symbols
\begin{lemma}\label{sono.horm}
	$\forall k \in \Z^d \setminus\{0\}$ one has
	$$
	\begin{gathered}
	\chi^T_k,\ (1 - \chir)(1 - \chi_k),\ \tilde\chi_k \in S_\varsigma^0\,, \\
	d_k^T \in S_\varsigma^{-\delta}\,,
	\end{gathered}
	$$ with seminorms uniformly bounded in
	$k$.
\end{lemma}
We omit the proof, which is a  variant of the proofs of  Lemmas
	6.3, 6.4, 6.6 of \cite{anarmonico}. 

Given $F \in \SF m$ {self-adjoint}, we use the above functions
to decompose $F$:
\begin{align}
\label{split.1}
&	F^{(\res)}_0:= \sum_{k \in \Z^d }\chi_k^{T}(A) \tilde\chi_k(A)\hat{F}_k\ ,
\\
\label{split.2}
& F^{(\nr)}_0:= \sum_{k \in \Z^d \setminus \{ 0 \}}(1-\chir(A))(1-\chi_k(A)) \tilde\chi_k(A)
\hat{F}_k\ ,
\\
\label{split.3}
& F^{(\Sm)}_0 := \sum_{k \in \Z^d \setminus \{ 0
	\}}(1-\chir(A))(1-\tilde\chi_k(A))\hat{F}_k+\chir(A)
F\ ,
\end{align}
and
\begin{align}
\label{iPezzi}
F^{(\res)}:=\frac{F^{(\res)}_0+(F^{(\res)}_0)^* }{2}\ ,\quad
F^{(\nr)}:=\frac{F^{(\nr)}_0 +(F^{(\nr)}_0)^* }{2}\ ,
\\
F^{(\Sm)}:=\frac{F^{(\Sm)}_0+(F^{(\Sm)}_0)^* }{2}\ ,
\end{align}
so that each one of the operators is self-adjoint and one has
$
F:=({F+F^* })/{2}
$, and therefore
$F = F^{(\nr)} + F^{(\res)} + F^{(\Sm)}
$.

\begin{remark}\label{sono.simboli}
	Let $\varsigma$ and $\rho$ be defined as in \eqref{varsigma}.
	If $F \in \rS{\rho}^m$ for some $m \in \R$, then $F^{(\res)}, F^{(\nr)} \in \rS{\rho}^m$.
\end{remark}

\begin{remark}
	\label{split}
	By Lemma \ref{lem.prod.n} {and Remark \ref{delA}} one has that
	$F^{(\nr)}=F^{(\nr)}_0+\rS\rho ^{m-\rho}\ .$
\end{remark}
%

Concerning $F^{(\Sm)}$, we have the following lemma.

\begin{lemma}
	\label{smoothing.l}
	Assume $F\in\rS\rho^m$, then $F^{(\Sm)}\in \rA\rho^{-\infty}$.
\end{lemma}
\proof First remark that the statement is obviously true for the 
second term of \eqref{split.3}.  Consider now the first term of
\eqref{split.3}. We are going to prove that $\forall s$ it is
smoothing of order $s-m$, from which the result
follows. First consider the case $m=0$; we prove that $F_0^{(\Sm)}$
maps $\cH^0$ to $\cH^{s}$ for all $s {\in \N.}$
Using the spectral decomposition of the operators
$A$, we have
\begin{align*}
F_0^{(\Sm)}\psi&=\sum_{a,k} {(1-\chir(A))(1-\tilde\chi_k(A))}\Pi_a\hat
F_k\psi\\
&=\sum_{a,k} {(1-\chir(a))(1-\tilde\chi_k(a))}\Pi_a\hat F_k\psi\ ,
\end{align*}
so that,  {using that $(1-\tilde\chi_k(A))$} is different from zero only if
$\|a\|^\mu<\|k\|$ {and applying Remark \ref{norme},} one has
\begin{align*}
\left\|F_0^{(\Sm)}\psi\right\|_s^2& {\leq c_1^{-s}}\sum_{a}\langle a\rangle^{2s}
\Big\|\sum_{k}  {(1-\chir(a))(1-\tilde\chi_k(a))}\Pi_a\hat F_k\psi\Big\|_0^2
\\
&\leq {c_1^{-s}} \sum_{a}\langle a\rangle^{2s}
\Big(\sum_{\|k\|>\|a\|^{\mu}} \Big\|\Pi_a\hat F_k\psi
\Big\|_0\Big)^2\\
&\leq {c_1^{-s}} \sum_{a}\langle a\rangle^{2s}
\Big(\sum_{\|k\|>\|a\|^{\mu}} \Big\|\hat F_k\Big\|_{\cB(\cH^0;\cH^0)}\left\|   \psi
\right\|_0\Big)^2\\
&\leq {c_1^{-s}} C\sum_{a}\langle a\rangle^{2s}
\Big(\sum_{\|k\|>\|a\|^{\mu}} \frac{ \langle
	k\rangle^N\wp^{0}_{\rho, j}(\hat F_k)} {\langle k\rangle^N}\|   \psi
\|_0\Big)^2
\\
&\leq {c_1^{-s}}        C\sum_{a}\langle a\rangle^{2s}\frac{1}{\|a\|^{2\mu N}}
\Big(\sum_{k} { \langle
	k\rangle^N\wp^{0}_{\rho, j}(\hat F_k)}\Big)^2 \left\|   \psi
\right\|_0 ^2\,.
\end{align*}
Taking $N>\frac{s+d}{\mu}$, one gets that the above quantity is
bounded by a constant times $\left(\wp^{0}_{\rho, j, N}(F)\right)^2 \left\|   \psi
\right\|_0 ^2$,
which proves the statement in the considered case. 
The general case
$m\not=0$ and $\psi\in {\cH^{s'}}$ is easily reduced to the previous one by
considering the operator $ {\langle A\rangle^{{-m +s'}}F\langle
A\rangle^{{-s'}}}$.
\qed

\begin{lemma}
	\label{nor.4}
	$F^{(\res)}$ is in normal form.
\end{lemma}
\proof One has 
\begin{equation*}
F^{(\res)}:= \frac{1}{2}\sum_{k \in \Z^d}\left[(1-\chir(A))
\chi_k(A) \tilde\chi_k(A)\hat{F}_{k}+\hat{F}_{-k}(1-\chir(A))
\chi_k(A) \tilde\chi_k(A)\right]\ .
\end{equation*}
Consider first the first term in the sum. By Lemma \ref{lego} and by
the definition of the cutoffs one has
\begin{align*}
\langle\psi_b;(1-\chir(A))
\chi_k(A) \tilde\chi_k(A)\hat{F}_{k}\psi_a\rangle= 
\langle(1-\chir(A))
\chi_k(A) \tilde\chi_k(A)\psi_b;\hat{F}_{k}\psi_a\rangle
\\
=(1-\chir(b))
\chi_k(b) \tilde\chi_k(b)\langle\psi_b;\hat{F}_{k}\psi_a\rangle
= (1-\chir(b))
\chi_k(b)
\tilde\chi_k(b)\delta^{a+k}_b\langle\psi_b;{F}\psi_a\rangle \ ,
\end{align*}
which means that this term does not vanish only when $b$ is resonant
with $k=b-a$, so this term is in normal form. An equal computation
shows that the second term is different from zero only if $a$ is resonant
with $k=b-a$, so that the lemma is proven. \qed

\noindent
{\it Proof of Lemma \ref{lem.hom}}. First, we remark that for
$\delta_*$ sufficiently close to $\mom=\td-1$ one has that
$\varsigma := 1 -(\mom-\delta)$ is close to $1$. Consequently,
since the function $\rho(\cdot)$ defined in Assumption \ref{assumption.A}.{iv} is
continuous, also $\rho(\varsigma)$ is close to $1$, and one
gets
$$
2\rho+\delta-\td = 2\rho + \varsigma - 2 > 0\,.
$$ Then we put $Z:=F^{(\res)}$ and we include
$F^{(\Sm)}$ in the remainder term $\cA_\rho^{-\infty}$. 
Then we define
\begin{align}
\label{G0}
G_0:=\sum_{k\not=0}{d^T_k(A)}\hat F_k\,, \quad G:=\frac{G_0+G_0^*}{2}=G_0+\rS\rho ^{m-\delta-\rho}\ ,
\end{align}
so that $G$ is self-adjoint.

Since by Lemma \ref{sono.horm} and Assumption \ref{assumption.A}.{iv}, ${d^T_k(A)} \in
\cA_\rho^{-\delta}$, one has $G_0 \in \cS_\rho^{m - \delta}$ and also 
$G \in \cS_\rho^{m - \delta}$
with
seminorms bounded by the seminorms of $F$.

We verify now that such a $G$ solves the cohomological equation. 
By the generalized commutator lemma
(Theorem \ref{commutator}) and Remark \ref{kj} one has
\begin{align*}
-\im [H_0;G]=-\im
[H_0;G_0]+\SF{m-\delta+\td-2\rho}=\sum_{j=1}^d-\im\frac{\partial
	h_0}{\partial a_j}(A)[A_j;G_0]+\rS\rho ^{m-\delta+\td-2\rho}
\\
= {\sum_{k \in \Z^d}}\sum_{j=1}^d\im \omega_j(A)k_j\hat G_{0,k}+\rS\rho ^{m-\delta+\td-2\rho}=\sum_{k\in\Z^d}^d(1-\chi^{(\tR)}(A))(1-\chi_k(A))\hat
 F_k+\rS\rho ^{{m-\delta+\td-2\rho}}
 \\
=F_0^{(\nr)}+\rS\rho ^{m-\delta+\td-2\rho} = F^{(\nr)}
+\rS\rho ^{{m-\delta+\td-2\rho}}\ , 
\end{align*}
since $F^{(\nr)}_0 - F^{(\nr)} \in \rS{\rho}^{m-\rho}$ and ${\rho> \rho+(\rho-1)+(\delta-M) = 2 \rho + \delta- \td.}$
\qed

\subsection{End of the proof of Theorem \ref{norm.form}}\label{finedim1}

We are now going to prove the following iterative lemma, which
immediately yields Theorem \ref{norm.form}

\begin{lemma}\label{norm.form.lemma}
	Let $H$ be as in equation \eqref{H.mai}, with $V \in
	\cinfinitos\rho \tb$, $\tb<\td$, a
	family of self-adjoint operators. 
	There exists $0<\delta_* <\mom$ such that, if $\delta_*<\delta<\mom$, then $\da$ defined as in \eqref{sigma.1} satisfies $\da>0$ and the following holds.
	For any $\mu, \tR$ and $\forall n\geq 0$ there exists a time-dependent
	family $U_n(t)$ of unitary maps conjugating the operator $H$
	of \eqref{H.mai} to
	\begin{align}
	\label{eq.in.forma}
	H_n(t)=H_0+Z_n(t)+  R_n(t)+\widetilde R_n(t)\,,
	\end{align}
	where:
	\begin{enumerate}
		\item $Z_n \in \cinfinitos\rho \tb$ is a family of self-adjoint operators in normal form
		\item $R_n \in \cinfinitos\rho {\tb- n\da}$ and $R_n(t)$ is a
		family of self-adjoint operators
		\item $\widetilde R_n \in C^\infty(\R;\cA^{-\infty}_\rho)$ and
		$\widetilde R_n(t)$ is a
		family of self-adjoint operators
		\item For any $s \geq 0,$ $U_N,\ U_N^{-1} \in
		{L^\infty}(\R;\cB(\cH^s;\cH^s))\,.$ 
	\end{enumerate}
\end{lemma}
\proof
First we observe that there exists $\delta_* = \delta_*(\tb)>0$ such that, if $\delta_*<\delta<\mom$, then $\da>0$.
We {prove the theorem by induction.} In the case $n=0$, the claim is trivially true taking
$U_0(t)=\uno$, $Z_0 = 0$, $R_0 = V$ and $\widetilde R_0=0$.

We consider now the case $n>0$.  Denote $m:=\tb-n\da$; we determine
$G_{n+1}\in\rar^{\eta}$, $\eta=\tb-n\da-\delta$, according to Lemma
\ref{lem.hom} with $F$ replaced by $R_n$. Up to increasing again the
value of $\delta_*$, {and using again Assumption \ref{assumption.A}.{iv},} one has that
$$
\eta = \tb - n\da -\delta \leq \tb - \delta = -(\td -\tb) + (\mom - \delta) + 1 <{\rho}\,.
$$

Then one uses $e^{\im G_{n+1}}$ to conjugate $H_n$ to
$H^+$ given by (see Lemma \ref{hpiu} with $m$ replaced by $\td$)
\begin{equation}\label{passaggetti}
\begin{aligned}
  H^+&=H_n-\im\left[H_n;G_{n+1}\right]{+\frac{1}{2}[[H_n;G_{n+1}];G_{n+1}]+
  \rar^{\td+3(\eta-\rho)}}+\rar^{\eta}+
\cA^{-\infty}_\rho
\\
&=
H_n-\im\left[H_0;G_{n+1}\right]+{\frac{1}{2}[[H_0;G_{n+1}];G_{n+1}]+\rar^{\tb+\eta-\rho}+
\rar^{\td+3(\eta-\rho)}}+\rar^{\eta} {+ \cA_{\rho}^{-\infty}}
\\
&=   H_0+Z_n+R_n^{(\res)}+\rar^{m-(2\rho+\delta-\td)}+\rar^{\tb+\eta-\rho}+
{\rar^{\td+3(\eta-\rho)}}+\rar^{\eta} {+ \cA_{\rho}^{-\infty}}\ ,
\end{aligned}
\end{equation}
where we go from the first to the second line by inserting the
expression of $H_n$ and from the second to the third line using Lemma
\ref{lem.hom}. 

Writing explicitly the different exponents of the  classes {of the remainder terms in the last line of \eqref{passaggetti},} we get that
they are given by
\begin{align*}
e_1:=\tb-n\da-(2\rho+\delta-\td)=\tb-n\da-\da_1\ ,\quad
\da_1:= 2\rho+\delta-\td
\\
e_2:=\tb+\tb-n\da-\delta-\rho=\tb-n\da-\da_2 \ ,\quad
\da_2:=\delta+\rho-\tb\ ,
\\
{e_3:=\td+3(\tb-n\da-\delta-\rho)=\tb-n\da-\da_3 \ ,\quad
\da_3:=-2(b-n\da)-\td+3\delta+3\rho}\ ,
\\
e_4:=\tb-n\da-\delta=\tb-n\da-\da_4\ ,\quad
\da_4:=\delta\ .
\end{align*}

Noting that ${\da_3 \geq\da_2}$ and choosing the smallest value of $\da$, the conclusion follows.
\qed

\section{Geometric part and conclusion of the proof}

\begin{definition}
	\label{proietto}
	Given a subset $E\subset\Lambda$ we define $\Pi_{E} := \sum_{a \in E} \Pi_a\,.$
\end{definition}

In the present section we prove that an operator in normal form  {leaves
invariant a dyadic partition, precisely we prove the following theorem.}

\begin{theorem}\label{brutto.ma.vero}
	Suppose that $H_0$ satisfies Assumption \ref{assumption.H}, 
	then there exists a partition $\{W_\ell\}_{\ell \in \N}$ of $\Lambda$
	with the following properties.
	\begin{itemize}
		\item[(1)] The sets $W_\ell$ are {finite and} dyadic, namely 
		\begin{equation}
		\max\{\|a\|\ :\ a \in W_\ell\} \leq 2 \min\{\|a\|\ :\ a \in W_\ell\} \quad \forall \ell\,.
		\end{equation}
		\item[(2)] The sets $W_\ell$ are invariant for any
                  operator $Z$ in  normal form, namely 
		\begin{equation}
		[\Pi_{W_{\ell}}, Z] = 0 \ ,\quad \forall \ell \in \N\,.
		\end{equation}
	\end{itemize}
\end{theorem}

This is the heart of the proof, indeed, by exploiting this Theorem the
proof of Theorem \ref{main} is concluded following exactly the
procedure of Sect. 5 of \cite{BLMgrowth}. {Let us summarize the strategy: given $A(t)$ a family of time dependent self-adjoint operators, let $\cU_A(t, \tau) $ be the time dependent unitary flow associated to $A$. 
First one proves that, due to the dyadic property of the blocks, $\|\cU_{H_0 + Z_N}(t, \tau) \psi\|_{s} \leq K_{s} \|\psi\|_{s}$ for all times $t, \tau$. This is Lemma 5.1 of \cite{BLMgrowth}. Then, by Duhamel formula, one obtains $\|\cU_{H^{(N)}}(t, \tau) \psi\|_{s} \leq K_{s, N} \langle t - \tau \rangle \|\psi\|_{s}$ for all $t, \tau$ (see Proposition 5.2 of \cite{BLMgrowth}) and by interpolation one gets \eqref{growth.eq} with $\cU = \cU_{H^{(N)}}$. Finally Theorem \ref{main} follows using boundedness in $\cH^s$ of the unitary maps $U_N$ of Theorem \ref{norm.form}.}

The construction of the sets $W_\ell$ is a generalization to the {abstract} steep case of that
done in \cite{BLMres, BLMgrowth} for quadratic $H_0$. It
represents a quantum counterpart of the geometric decomposition at the
basis of the proof of Nekhoroshev theorem.
The key point to have in mind is that an operator in normal form only
connects points $a$ and $b$ of the lattice $\Lambda$ s.t.\ $b-a$
is either resonant with $a$ or resonant with $b$ {(see Lemma \ref{nor.4} and its proof)}. 

\begin{remark}\label{rmk.importante}
	The whole construction that will be developed in this section is
        based only on the structural properties of $H_0$. In
        particular Theorem \ref{brutto.ma.vero} holds also for any
        linear self-adjoint operator $Z$ which satisfies
        \eqref{nor.for.def} and not only for pseudodifferential
        operators.
\end{remark}
{We recall the following:}
\begin{definition}\label{def.modulo}
	A subgroup $M$ of $\Z^d$ is said to be {pure}\footnote{{This notion is sometimes referred to in the literature on Hamiltonian systems as \emph{resonance module}, see for instance \cite{gio, nek_noi}.}} if
        $\Span_\R\{M\} \cap \Z^d = M$. 
\end{definition}

\subsection{The invariant partition: definitions and statement}

 {The construction depends on {the parameters} $\delta, \mu$
   {(controlling the notion of resonant point, see Definition
     \ref{res})} and on some further parameters. In order to get the
   result they have to fulfill some relations that we anticipate
   here. We assume}

\begin{equation}\label{legami.ep}
\alpha d(d-1) \mu < \mom - \delta\,, \quad \alpha := \alpha_1 \cdots \alpha_{d-1}\,,
\end{equation}
where the steepness indexes $\alpha_1, \dots, \alpha_{d-1}$ have been introduced
in Definition \ref{steep.def}. We also take positive parameters $1 =
\tC_1 < \tC_2 \cdots < \tC_d$, $
1 = \tD_1 < \tD_2 \cdots < \tD_d$ and 
 define
\begin{equation}
\label{parameters}
\gamma_1 := \mom - \delta\,, \quad \gamma_{s+1} := \frac{\gamma_s - s \mu}{\alpha_s} \quad \forall s = 1, \dots, d-1\,.
\end{equation}
\begin{remark}\label{rmk.b.posto}
	Equation \eqref{legami.ep} ensures that the parameters $\gamma_1, \dots, \gamma_{d}$ are positive, with $\gamma_1 > \cdots >\gamma_d.$
\end{remark}

In the following, neighborhoods of the origin will not be relevant, so
we will only consider the set 
\begin{equation}
\label{largenooo}
\|a\|\geq
\frac{1}{2}\ .
\end{equation}


{\emph{We fix now an open {convex} cone $\cC_e$ s.t.\ $\cC_e\supset \overline
		\cC\setminus\left\{0\right\}$ and {$\cC_e \setminus\{0\}$ is contained in the cone generated by $\cU$, namely
			$$
			{\cC_e \setminus\{0\} \subseteq \{a \in \R^d\ |\ \exists \lambda \in \R^+, u \in \cU \textrm{ s.t.\ } a = \lambda u\}\,,}
			$$}
with $\cU$ as in Assumption \ref{assumption.H}.ii.}}

	By the steepness assumption, on the ball $\left\|a\right\|=1$, $\omega(a)$ is
        bounded from below and from above and thus,
        by homogeneity, $\exists K>0$ such that
	\begin{equation}\label{om.hom.3}
	\| a\|^{\mom} K^{-1} \leq \| \omega(a)\| \leq K \| a
        \|^{\mom}\ ,
	\quad \forall a \in \R^d{\cap\cC_e}\ .
	\end{equation}

\begin{remark}
  \label{steep.scala}
	By homogeneity, the steepness condition \eqref{steep} implies that
	$\forall r\geq1$ and 
	$\forall a\in \left(B_{2r}\setminus B_{r/2}\right)\cap \cC_e$, one has 
	\begin{equation}
	\label{steep.r}
	\max_{0\leq\eta\leq\xi}\min_{u\in M:\norm u=1}\norma{\Pi_M
		\omega(a+\eta u)}\geq \tB_s r^{\mom-\alpha_s}  \xi^{\alpha_s} \quad \forall \xi \in (0, r\tr]\ ,
	\end{equation}
	This is the condition we will use in the proof.
\end{remark}

Keeping in mind Definition \ref{def.nf}  {of normal form}, one has that, if $Z$ is in normal form, then
\begin{equation}\label{stellina}
\exists \psi\ \quad s.t.\ 	\langle \Pi_a \psi; Z \Pi_{a + k} \psi \rangle \neq 0
\end{equation}
only if either $a$ is resonant with $k$ or $a + k$ is resonant with $k$ according to Definition \ref{res}.

 We start by identifying the points $a \in
\Lambda$ which are in resonance with the vectors of a one dimensional
{pure subgroup} $M$. They are the points  such that \eqref{stellina} holds for some $k \in M$, namely the points that a normal form operator moves in the directions of $M$.} For this reason, we give the following definition:
\begin{definition}\label{1.res}
	Given $M$ a {pure subgroup} of dimension $1$ and $k \in M$, we define for $\sigma \in \{0, 1\}$
	\begin{equation}
	\cZ_k^\sigma := \left\{ a \in \Lambda \ |\ a + k \in \Lambda\ \text{ and }\ a + \sigma k\ \text{ is resonant with }k \right\}
	\end{equation}
	and
	\begin{equation}
	\cZ^{(1)}_M := \bigcup_{k \in M} \left(\cZ^0_k \cup \cZ^1_k\right)\,.
	\end{equation}
\end{definition}

\begin{remark}
	Points $a \in \cZ^0_k$ fulfill in particular ${|\omega(a) \cdot k| \leq \|k\| \|a\|^{\delta}\,,}$ whereas points $a \in \cZ^1_k$ fulfill $|\omega(a+k) \cdot k| \leq \|k\| \|a+k\|^{\delta}\,.$
\end{remark}

We define now the nonresonant points as the points in $\Lambda$ that
belong to none of the resonant zones $\cZ^{(1)}_M$.

\begin{definition}[Nonresonant set] \label{non.res} We define
        	$$
	\begin{gathered}
	\Omega_0 := \{a \in \Lambda \ |\ a \textrm{ is nonresonant with }b-a \quad \forall b \in \Lambda\}\,,\\
	\Omega_1 := \{a \in \Lambda \ |\ b \textrm{ is nonresonant with
	}b-a \quad \forall b \in \Lambda\}\, .
	\end{gathered}
	$$
The set $	\Omega = \Omega_0 \cap \Omega_1$. 
is called \emph{nonresonant set}.
\end{definition}

\begin{remark}\label{seconda.parte.inutile} 
	One has
	\begin{equation}
	\Omega_\sigma := \Lambda \setminus \left( \bigcup_{k \in \Z^d} \cZ^{\sigma}_k \right)\,.
	\end{equation}
	\end{remark}

We come to multiple resonances.

\begin{definition}\label{def.sres} \null 
 Let $j = 1, \dots, d$. Given  $a \in \Lambda$ and $k \in
        \Z^d$, we say that \emph{$a$ is resonant with $k$ at order $j$} if 
	\begin{equation}\label{res.k}
	\|a\| \geq \tR\quad \wedge \quad \|k\| \leq \tD_j \|a\|^\mu
	\end{equation}
	and
	\begin{equation}\label{res.om}
	|\omega(a)\cdot k | \leq \tC_j \|k\| \|a\|^{\mom-\gamma_j}\,.
	\end{equation}
\end{definition}
  When $j=1$, since $\tC_1 = \tD_1 = 1$ {and $M -\gamma_1 =\delta$},
  one recovers the Definition \ref{res}.

\begin{definition}[Resonant zones]\label{RZ}
	Let $M$ be a {pure subgroup} of $\Z^d$ of dimension $s$.
	\begin{itemize}
		\item[(i)] If $s = 0$, namely $M = \{0\},$ we set $	\cZ^{(0)}_M=	\cZ^{(0)}_{\{0\}} := \Omega$.
		\item[(ii)] If $s\geq 1$, for any set of linearly
		independent vectors $\{k_1, \dots, k_s\}$ in $M$, we
		define, for $\sigma=0,1$
		\begin{align*}
&\cZ^\sigma_{k_1, \dots, k_s} = \{ a \in
\Lambda\ |\
\\
&\left(a + k_1 \in \Lambda
                \right)\ \ \wedge\ \ \left(a+ \sigma k_1 \textrm{ is }
                \textrm{resonant with } k_j\ \textrm{at order } j  \quad \forall j = 1, \dots, s\right) \}\,,
                \end{align*}
				and
		\begin{equation}
		\cZ^{(s)}_{M} := \bigcup_{\begin{subarray}{c}
			k_1, \dots, k_s \\ \textrm{lin. ind. in } M
			\end{subarray} } \left(\cZ^0_{k_1, \dots, k_s} \cup \cZ^1_{k_1, \dots, k_s} \right) \,.
		\end{equation}
	\end{itemize}
	The sets $\cZ^{(s)}_M$ are called \emph{resonant zones}.
\end{definition}
We point out that, if $s=1$, we recover Definition \ref{1.res}.
{Strictly speaking the points in $\cZ^{(0)}_{\left\{0\right\}}$ are non
resonant, however, for compactness of language it is convenient to
call them $0$-resonant.}
\\
The sets $\cZ^{(s)}_{M}$ contain lattice points $a$ which are in
resonance with \textit{at least} $s$ linearly independent vectors in
$M$.

\begin{lemma} \label{rmk.inscatolate}
	Fix $r, s \in \{1, \dots, d\}$ with $1\leq r<s$, then for any $M$ with dim $M=s$, one has
	$$
	\cZ^{(s)}_{M} \subseteq \bigcup_{ \begin{subarray}{c} M^\prime \subset M\\
		\textrm{dim.} M^\prime = r 
		\end{subarray}} \cZ^{(r)}_{M^\prime}\,.
	$$
\end{lemma}
	The result follows from the very definition of the resonant
        zones, the details are left to the reader.


{
	We start the study of the properties of the resonant and nonresonant zones.  {First, one has that, by the first of
          \eqref{res.k} points of $\Lambda$ with a small norm are
          considered as nonresonant. A quantitative statement is the
          content of the next lemma.}

        \begin{lemma}\label{origina}
		Assume $\tR$ large enough, then 
		$$
		B_{\tR/2} \subset \Omega\,.
		$$
	\end{lemma}
	\proof Given $a \in \Lambda,$ suppose that there exist $j = 1,\dots, d$, $\sigma \in \{0, 1\}$ and $k \in \Z^d$ such that $a + k \in \Lambda$ and $a + \sigma k$ is resonant with $k$ at order $j$. When $\sigma = 0$, \eqref{res.k} gives $\|a\| \geq \tR$, therefore $\Omega_0 \supset B_\tR$. On the other hand, for $\sigma = 1$, by \eqref{res.k} with $a$ replaced by $a + k$ one has $\| a + k\| \geq \tR$, and
	$$
	\|a\| \geq \| a +  k \| - \| k \| \geq \| a + k \| -  {\tD_j}\|a + k \|^{\mu}\,,
	$$
	but the right-hand side turns out to be larger than $\tR/2$,
	provided $ {\left\|a+ k \right\|}$ is large enough.
	This gives $\Omega_1 \supset B_{\frac{\tR}{2}}$, and therefore
	$\Omega = \Omega_0 \cap \Omega_1 \supset B_{\frac{\tR}{2}}$. \qed
	\begin{remark}\label{rmk.res.big}
	Since the resonant zones $\cZ^{(s)}_M$ {with $s>0$} and $\Omega$ are disjoint sets, by Lemma \ref{origina} one has that $a \in \cZ^{(s)}_M$ for some  $M$ implies $\|a\| \geq \frac{\tR}{2}\,.$
	\end{remark}

In the next lemma, whose proof is deferred to Subsection \ref{sec.bd}, we claim that the ``completely resonant zone'' is empty:

\begin{lemma}\label{lemma.bd.1}
	Provided $\tR$ is large enough, the resonant zone
	$\cZ^{(d)}_{\Z^d}$ is empty.
\end{lemma}


{We now introduce an equivalence relation, which encodes the fact that two sites $a, b \in \Lambda$ can be connected one each other by an operator in normal form only if they are not too distant. The equivalence classes play in the present construction the same role that the different connected components of the resonant zones play in classical Nekhoroshev theorem.}

\begin{definition}\label{pre}
	On $\cZ^{(s)}_M$ we define the following pre-equivalence relation: $a \sim' b$ if $a - b \in M$ and
	\begin{equation}\label{connect}
	\| a - b\| \leq \max\{{\|a\|}^{\mu},\, {\|b\|}^{\mu}\}\,.
	\end{equation}
	We then complete such a
	pre-equivalence relation to an equivalence relation setting
	$a \sim b$ if there exists a {finite} sequence $\{a_j\}_{j = 1}^{N}$ such that $a_1 = a$, $a_N = b$ and ${a_j \sim' a_{j+1}}$ for any $j = 1, \dots N-1$.
\end{definition}
\begin{definition}\label{equi}
  For any {pure subgroup} $M$ of dimension $s$ there could be several
  equivalence classes that will be labeled by an index $j=1,...,j_M$,
  {with $j_M\leq \infty$,}
  we will denote $\cJ_M:=\left\{1,...,j_M\right\}$ and by 
	$A^{(s)}_{M, j}$ the $j$-th equivalence classes
	with respect to the {equivalence relation $\sim$.}
\end{definition}

\begin{remark}\label{rmk.M0}
	In the case $M = \{0\}$, the condition $a - b \in M$ implies that each equivalence class $A^{(0)}_{\{0\}, j}$ contains only one element of $\cZ^{(0)}_{\{0\}}$, and  $\cJ_M = \sharp \cZ^{(0)}_{\{0\}}$.
\end{remark}
The following lemma encodes the main properties of the equivalence classes. In particular, Items $1.$ and $3.$ ensure that points belonging to the same equivalence class are connected by a vector in $M$ and are not too far each other.
\begin{lemma}\label{lemma.diametri}
	The equivalence classes $\{A^{(s)}_{M, j}\}_{j \in \cJ_{M}}$ fulfill the following properties:
	\begin{enumerate}
	\item $a, b \in A^{(s)}_{M, j}$ $\Rightarrow a - b \in M$
	\item If $a \in A^{(s)}_{M, j_1}$ and $b \in
	A^{(s)}_{M, j_2}$ with $j_1 \neq j_2$, then either $\| a - b\| > \max\{ {\|a\|}^{\mu},\, {\|b\|}^{\mu}\}$, or $a -b \notin M$\,.
	\end{enumerate}
	Furthermore, provided $\tR$ is large enough, there exists a positive constant
	$C$, depending only on 
	$\gamma_s$, $\tC_s$ and $\tD_s$, such that the following holds:
	\begin{enumerate}
	\setcounter{enumi}{2}
	\item $a, b \in A^{(s)}_{M, j}$ $\Rightarrow$ $\| a-b\| \leq C {\|a\|}^{1 - \gamma_{s+1}}\,.$
	\end{enumerate}
\end{lemma}
Since the proof of Item $3.$ is nontrivial, we postpone the proof of this lemma to Subsection \ref{sec.amj}.

Clearly the regions $\cZ^{(s)}_{M}$ are not reciprocally disjoint. Following the construction of \cite{BLMres, GCB, nek_noi}, we identify
now sets of points $a \in \Lambda$ which admit \textit{exactly} $s$ linearly
independent resonance relations.
\begin{definition}[Resonant blocks] \label{def blocks}\ 
	\begin{enumerate} 
		\item (Nonresonant blocks $B^{(s)}_{M,j}$): If $M = \{0\}$,  
		$$ B^{(0)}_{\{0\}, j}:= A^{(0)}_{\{0\}, j} \,\quad ,
		\forall j \in \cJ_{\{0\}}
		$$
		\item ($s$-resonant blocks): Given a {pure subgroup}
		$M \subset \Z^d$ of dimension $s \in \{ 1, \dots,
		d-1\},$ we define  
		$$
		B^{(s)}_{M, j} := A^{(s)}_{M, j}  \backslash \left
		\lbrace \bigcup_{\begin{subarray}{c} M^\prime
			\textrm{s.t.} \dim M^\prime = s+1 \end{subarray}}
		\cZ^{(s+1)}_{M^\prime} \right \rbrace\,,\quad j \in
		\cJ_{M}\ ,
		$$
	\end{enumerate}
	where $\cJ_M$ and $\{A^{(s)}_{M, j}\}_{j \in \cJ_M}$ are the
	sets whose existence is ensured in Lemma
	\ref{lemma.diametri}. 
\end{definition}
\begin{remark}
	\label{nonio}
	The resonant blocks form a covering  of $\Lambda$.
\end{remark}
As we will prove in the following section, $\forall s$ there exists a
suitable choice of the constants $\tC_s,\ \tD_s$ such that the blocks $B^{(s)}_{M, j}$ are reciprocally disjoint.
However, as in the classical case, they are not invariant under the
action of a normal form operator, so we define the quantum extended blocks, which instead will be invariant.
First, given two sets $A $ and $B,$ we define:
$$
A + B := \{ a + b \  | \ a \in A\,,\ b \in B
\}\ .
$$
\begin{definition}[Extended blocks $E^{(s)}_{M,j}$] \ 
	\begin{enumerate} 
		\label{def blocchi estesi}
		\item $\displaystyle{E^{(0)}_{\{0\},
				j}:=B^{(0)}_{\{0\}, j} \equiv A^{(0)}_{\{0\}, j}}$
		$j \in \cJ_{\{0\}}$
		\item Given a {pure subgroup} $M$ of dimension $1$, $\forall j \in \cJ_M$ we define
		$$
		E^{(1)}_{M, j} := \left\{ B^{(1)}_{M, j} + M \right\} \cap A^{(1)}_{M, j}\,,
		\quad		E^{(1)} := \bigcup_{M \textrm{ of dim.} 1,\, j \in \cJ_{M} }
		E^{(1)}_{M, j}\,.
		$$
		\item Given a {pure subgroup} $M$ of dimension $s$,
		with 
		$2\leq s < d,$ for any $j \in \cJ_{M}$
		we define 
		$$
		E^{(s)}_{M, j} := \left\{ B^{(s)}_{M, j} + M \right\} \cap
		A^{(s)}_{M, j} \cap 
		{\bigcap_{k=1}^{s-1}
			\left(E^{(s-k)}\right)^{c}\,,}\quad E^{(s)} := \bigcup_{ M \textrm{ of dim.}s\,,\, j \in \cJ_M } E^{(s)}_{M, j}\,,
		$$
		where, given $E \subseteq \Lambda,$
		$E^c:=\Lambda\setminus E$.
	\end{enumerate}
\end{definition}

\begin{remark}
	The extended blocks $\{E^{(s)}_{M, j}\}$ are still a covering
	of $\Lambda$. 
\end{remark}

The following theorem is the main result of the
present section, and it is the heart of the geometric part of the proof. Its proof is postponed to Subsection \ref{6.17}:

\begin{theorem}\label{teo.invarianza}
	There exists a choice of the parameters $\tC_1,\dots, \tC_d$,
	$\tD_1, \dots, \tD_d$ and $\tR$ such that the blocks
	$\{E^{(s)}_{M, j}\}_{M \subset \Z^d,\,j \in \cJ_M}$ are a
	partition of $\Lambda,$ which is dyadic and is left invariant by operators
	$Z$ which are in normal form.
\end{theorem}


\subsection{Proof of Lemma \ref{lemma.bd.1}}\label{sec.bd}

As in the proof of the classical Nekhoroshev theorem (see also \cite{BLMres}), the following
Lemma from \cite{gio}, to which we refer for the proof, plays a fundamental role.

\begin{lemma}[Lemma 5.7 of \cite{gio}]\label{lemma.giorgilli}
	Let $s \in \{1\,, \dots\,, d\}$ and let $\{ u_1\,, \dots u_s
	\}$ be linearly independent vectors in $\R^d\,.$ Let $w \in
	\Span{\{ u_1\,, \dots u_s\}}$ be any vector. If $\alpha\,, N$
	are such that
	\[
	\begin{gathered}
	\norm{u_j} \leq N \quad \forall j = 1\,, \dots s\,,\\
	| \scala{w}{u_j} | \leq \alpha \quad \forall j = 1\,, \dots s\,,
	\end{gathered}
	\]
	then
	\[
	\norm{w} \leq \frac{s N^{s-1} \alpha }{\operatorname{Vol} \{u_1\,| \cdots\,| u_s \}}\,.
	\]
\end{lemma}

\begin{proof}[Proof of Lemma \ref{lemma.bd.1}]
	Assume that $\cZ^{(d)}_{\Z^d}$ is not empty and take $a \in
	\cZ^{(d)}_{\Z^d}$. First, by Remark \ref{rmk.res.big}, one has
        $\|a\| \geq \tR/2$.        
	Furthermore there exists $\sigma \in \{0, 1\}$ and $\{{k}_1, \dots,
	{k}_{d}\} \subset \Z^d$ linear independent vectors such that $\forall
	j$
	\begin{equation}\label{all.small}
	\begin{gathered}
	\norm{{k}_j}\leq \tD_{j} \| a + \sigma k_1 \| ^{\mu} \leq \tD_{d} \| a + \sigma k_1 \|^{\mu}\,,\\
	\begin{aligned}
	| \scala{ \omega(a + \sigma k_1)}{{k}_j}| &\leq \tC_{j} \| k_j\| \| a + \sigma k_1\|^{\mom -\gamma_{j}} \leq \tC_{d} \tD_{d} {\|a + \sigma k_1\|}^{\mom -\gamma_{d} + \mu}\,.
	\end{aligned}
	\end{gathered}
	\end{equation}
	Consider first the case $\sigma = 0$. By the second of
	\eqref{all.small}, using Lemma \ref{lemma.giorgilli} we deduce
	$$
	\| \omega(a)\| \leq d (\tD_d)^{d} \tC_d {\|a\|}^{\mom- \gamma_d  + \mu d}\,.
	$$ Then by Equation \eqref{om.hom.3}, one
	has
	\begin{align}\label{om.hom}
	K^{-1} \| a\|^{\mom} \leq \| \omega(a)\| \leq d (\tD_d)^{d} \tC_d {\|a\|}^{\mom-\gamma_{d} + d \mu}\,.
	\end{align}
	Recalling that $\displaystyle{\mom -\gamma_{d} + d \mu < \mom,}$
	\eqref{om.hom} implies $\| a\| \leq \tR_0$ for some positive $\tR_0$,
	but if $\tR/2 > \tR_0$ this is in contradiction with the fact that $\|a\| \geq \frac{\tR}{2}$.\\ If instead $\sigma = 1,$ in order to eliminate the
	presence of the vectors $k_1$ in estimates \eqref{all.small}, we apply
	Lemma \ref{incrementini} of the appendix, with $k = k_1$, $h = k_j$,
	$a = b$, $l = 0$, and then conclude the proof as in the case
        $\sigma=0$. 
\end{proof}

\subsection{Proof of Lemma \ref{lemma.diametri}} \label{sec.amj}
Items $1$ and $2$ of Lemma \ref{lemma.diametri} are easily verified.
Due to Remark \ref{rmk.M0}, Item $3$ also immediately follows in the
case $M = \{0\}$. We now tackle the case $M \neq \{0\}$; this is the
heart of the proof of Theorem \ref{brutto.ma.vero}, in particular it
is where steepness comes into play.

In the remaining part of the present subsection, \emph{we
	will restrict to the case where $M$ is a proper {nonzero} {pure subgroup of $\Z^d$}}, namely
$1\leq \operatorname{dim} M\leq d-1$.   
\begin{lemma} \label{lemma piccole proiez}
	If $a \in \cZ^{(s)}_M$, then there exists a positive constant $K$ depending only on $d, \mu, \gamma_{s}, \tC_{s}, \tD_{s},$ such that 
	\begin{equation}
	\label{proi.1}
	\|\Pi_M \omega(a)\| \leq K {\|a\|}^{\mom - \gamma_{s}+s\mu}\,.
	\end{equation}
\end{lemma}
{\begin{proof}
	One argues as in the proof of Lemma \ref{lemma.bd.1}.
	By Definition \ref{RZ} of $\cZ^{(s)}_M,$ if $a \in \cZ^{(s)}_M$ there exist $k_1, \dots, k_s$ and $\sigma \in \{0, 1\}$ s.t.  for all $j = 1, \dots, s$
	\begin{equation}\label{dove.siamo}
	|\omega(a + \sigma k_1) \cdot k_j| \leq \tC_s \|k_j\| \| a + \sigma k_1\|^{\tM - \gamma_s}\,, \quad \|k_j\| \leq D_s \|a + \sigma k_1\|^{\mu}\,.
	\end{equation}
	Then if $\sigma=0$, one observes that $\left|\left(\Pi_M \omega (a)\right) \cdot k_j\right| = | \omega(a) \cdot k_j|$ for all $j$, and applies Lemma \ref{lemma.giorgilli} with $w = \Pi_M \omega(a)$ and $u_j = k_j$ to deduce $\|\Pi_M \omega(a)\| \leq s (\tD_s)^{s-1} \tC_s \|a\|^{\tM - \gamma_s + \mu s}$. If instead $\sigma = 1$, one applies Lemma \ref{incrementini} with $k=k_1$, $h=k_j$ $a = b$, and $l=0$, to obtain estimates of the form \eqref{dove.siamo} involving only  $a$ instead of  $a + \sigma k_j$, and then applies Lemma \ref{lemma.giorgilli} as in the case $\sigma = 0$.
	\end{proof}
}
	In order to be able to use steepness, one has to ensure that, given two points belonging to the same equivalence class, there always exists a continuous curve joining them such that, for any point in its support, a property analogous to \eqref{proi.1} holds:
\begin{lemma}[Interpolation]\label{lemma.curva}
	{For any $s = 1, \dots, d-1$,} there exists a positive constant $C$ {depending only on $d, \mu, \gamma_{s}, \tC_{s}, \tD_{s},$} such that the following
	holds. For any $M,j$ and for any $a, b \in A^{(s)}_{M, j},$
	there exists a curve $\gamma:[0, 1] \rightarrow \left(\{a\} +
	\Span_{\R} M\right) \cap \cC$ such that
	\begin{equation}\label{punti.giusti}
	\gamma(0) = a\,, \quad \gamma(1)= b\,,
	\end{equation}
	and 
	\begin{equation}\label{pm.small}
	\|\Pi_{M} \omega(\gamma(t))\| \leq C \langle \gamma(t)\rangle^{\mom-\gamma_s + s \mu} \quad \forall t \in [0, 1]\,.
	\end{equation}
\end{lemma}
\begin{proof} 
  Suppose first that $a$ and $b$ are such that $a\sim'b$, and
define $\gamma$ by $\gamma(t) =
	a + t(b-a)\subset\cC$. Furthermore,  by Remark
	\ref{rmk.a.or.b} there exists a positive constant $C$ such
	that
	\begin{equation}\label{le.i.stess}
	C^{-1} { \| a\| <\| a + t (b-a)\| \leq  C \|a\| \quad \forall t \in [0,\,1]\,,}
	\end{equation}
	so that (by the $(\mom-1)$-homogeneity of
	$\partial\omega/\partial a$):
	$$
	\begin{aligned}
	\| \Pi_{M} \omega(a + t(b-a))\| &\leq \| \Pi_{M} \omega(a)\| + \left\| \int_{0}^{1} \frac{\partial \omega}{\partial a}(a +t_1 t (b-a)) t(b-a)\,d t_1\right\|\\
	& \lesssim_{} {\|a\|}^{\mom-\gamma_s + s \mu} + {\|a\|}^{\mom-1 + \mu} \lesssim {\|a\|}^{\mom-\gamma_s + s \mu}\,,
	\end{aligned}
	$$
	where we used Lemma \ref{lemma piccole
		proiez}. Then, using again \eqref{le.i.stess},
	one also obtains
	$$
	\| \Pi_{M} \omega(a + t(b-a))\| \lesssim {\| a + t(b-a)\|}^{\mom- \gamma_s + s \mu} \quad \forall t\,.
	$$
	The general case $a \sim b$ follows by exploiting the
        previous result.
\end{proof}
Roughly speaking, the idea in order to prove Item 3 of Lemma
\ref{lemma.diametri} is that one would like to consider the curve
joining $a$ and $b$ and to exploit steepness in
order to deduce that, if by contradiction Item 3 of Lemma
\ref{lemma.diametri} does not hold true, then \eqref{pm.small} is violated. This is
obtained essentially as in the classical case. To this end, we first
need a few technical preparation Lemmas.

\begin{definition}\label{m.omega}
	For any $a \in \cZ^{(s)}_M$, define $\omega^\bot(a)$ as the $d-1$ dimensional subspace of $\R^d$ orthogonal to $\omega(a)$, and
	\begin{equation}
	\Momega := \Pi_{\omega^\bot(a)} \Span_{\R} M\,.
	\end{equation}
\end{definition}

The following lemma ensures that in an appropriate sense $M_a$ is
close to $M$. 

\begin{lemma}\label{lemma.pm.pmom}
	Let $M$ be a {pure subgroup} of dimension $s$, $1 \leq s<d$. There exists positive constants $C$ and $\tR_0$, depending only on $\gamma_s$, $\mu$, $\tC_s$, $\tD_s$, such that, if $\tR >\tR_0$, then for any $a \in \cZ^{(s)}_M$,
	\begin{equation}
	\label{6.27.1}
	\| \Pi_{\Momega} - \Pi_{M} \| \leq C {\|a\|}^{-\gamma_s  + s\mu}\,. 
	\end{equation}
\end{lemma}
\begin{proof}
	By Lemma \ref{lemma piccole proiez}, there exists a positive
        constant $C_1$, depending only on $\tC_s, \tD_s, \gamma_s, \mu$,
        such that for any $a \in \cZ^{(s)}_M$ we have 
	$$
	\| \Pi_{M} \omega(a)\| \leq C_1 {\|a\|}^{\mom-\gamma_s + s \mu}\,.
	$$
	Then Lemma \ref{lemma.proiettori} of the appendix and Equation
        \eqref{om.hom.3} give
	\begin{equation}\label{caffe}
	\| \Pi_{\Momega} - \Pi_M \| \leq 9 C_1 {\|a\|}^{\mom -
		\gamma_s + s \mu} \| \omega(a)\|^{-1}\leq 9 C C_1 {\|a\|}^{- \gamma_s + s \mu}
	\,
	\end{equation}
(with $C$ the constant in eq. \eqref{om.hom.3}),	provided $	\varepsilon:=  C_1 {\|a\|}^{\mom-\gamma_s + s \mu} \leq \frac{1}{2} \|\omega(a)\|$,
	but this inequality holds provided $\tR$ is large enough.
\end{proof}

\begin{lemma}\label{lem.piem.om.above}
	Given a {pure subgroup} $M$ of dimension $s$, let $a \in \cZ^{(s)}_M$
	and suppose {$ u \in \Span_{\R} M$ and $a + u \in
	\cC_e$} is
	such that $\exists C,\tau$ s.t.
	\begin{gather}
	\label{is.close}
	\| \Pi_M \omega(a + u)\| \leq C {\|a + u\|}^{\mom - \gamma_s +
          s \mu}\,,\quad 
	\| u\| \leq C {\|a\|}^{1-\tau}\,.
	\end{gather}
	Assume that 
	\begin{equation}
	\label{6.35}
	{a + \Pi_{\Momega}u \in \cC_e\ ,}
	\end{equation}
	then there exists positive constants $C^+$ and $\tR_0,$
        {both} depending only on $\gamma_s, \mu, \tC_s, \tD_s, C$ and $\tau,$ such that if $\tR > \tR_0$ one has
	\begin{equation}\label{piem.small}
	\left\| \Pi_{\Momega} \omega \left( a + \Pi_{\Momega} u \right) \right\| \leq C^+ {\|a\|}^{\mom - \gamma_s + s\mu}\,.
	\end{equation}
\end{lemma}
\begin{proof}
	One has
	\begin{align}
	\label{spostamenti.1}
	\Pi_{\Momega} \omega(a + \Pi_{\Momega} u) &=
	\Pi_{M} \omega( a +  u) + \left(\Pi_{\Momega} - \Pi_M\right) \omega( a +
	\Pi_{\Momega} u) \\
	\label{spostamenti.3}
	& + \Pi_{M} \left(\omega( a + \Pi_{\Momega}  u) - \omega( a +  u)\right)\,.
	\end{align}
  By
	\eqref{is.close}, one 
	has
	\begin{align}
	\label{primo.pezzo}
	\|\Pi_{M} \omega( a +  u)\| &\lesssim \langle a + u \rangle^{\mom-\gamma_s+s\mu }
	\\
	&\lesssim {\|a\|}^{\mom-\gamma_s+s\mu } + {\|u\|}^{\mom-\gamma_s+s\mu } \lesssim{\|a\|}^{\mom-\gamma_s+s\mu }\,,   
	\end{align}
	since $1-\tau< 1$.
	
The second term at right-hand side of Eq.\ \eqref{spostamenti.1} is estimated using again homogeneity and
	Eq.\ \eqref{6.27.1}. 	
	We come to 
	\eqref{spostamenti.3}. Recalling that $u \in \Span_{\R} M$, one has
	\begin{multline}
	\label{om.tayl}
	\|\Pi_{M} \left(\omega( a + \Pi_{\Momega} u) - \omega( a +  u)\right)\|  = \|\Pi_{M} \left(\omega( a + \Pi_{\Momega} u) - \omega( a + \Pi_M  u)\right)\|\\
	\leq \left\| \int_{0}^{1} \frac{\partial \omega}{\partial
		a}\left( a +
	 t (\Pi_{\Momega} - \Pi_M) u)\right)\,d t \right\| \left\| \left(\Pi_{\Momega} - \Pi_M\right) u\right\|\,.
	\end{multline}
	Using again Lemma \ref{lemma.pm.pmom}, one concludes the proof.
\end{proof}

We use now steepness in order to prove the following lemma, which concludes the proof of Lemma \ref{lemma.diametri}.

\begin{lemma}\label{lemma.ci.siamo.quasi}
	There exist positive constants $\overline{C}$ and $\tR_0$,
        depending on $\gamma_s, \mu, \tC_s, \tD_s, \tr$ only, such that if $\tR \geq \tR_0,$ $\forall a \in A^{(s)}_{M, j}$ one has
	\begin{equation}\label{onde}
	\| a-b\| \leq \overline{C} {\|a\|}^{1 - \gamma_{s+1}} \quad \forall b \in A^{(s)}_{M, j}\,.
	\end{equation}
\end{lemma}
\begin{proof}
{	We prove that there exist {large} constants $\overline R$ and
	$\overline{C}>0$, depending on $\tr, \gamma_s, \mu_s, \tD_s, \tC_s$
	only, such that if there exists a couple of points $a, b \in
	A^{(s)}_{M, j}$ satisfying $\| a\| \geq \overline R$ and
	\begin{equation}\label{far.far.away}
	{\| a - b\| \geq \overline{C} {\|a\|}^{1-\frac{\left(\gamma_s - s \mu\right)}{\alpha_s}},}
	\end{equation}
	one gets a contradiction. Then the result will follow taking
	$\tR \geq 2 \overline R$ and using Remark
	\ref{rmk.res.big}.}

Let us fix $a$ and $b$ in some $A^{(s)}_{M,
  j}$ and suppose that \eqref{far.far.away} holds.
Consider
	the curve $\gamma_{a, b}$ joining $a$ and $b$ constructed in
	Lemma \ref{lemma.curva}; then there exists $t^*>0$ such that
	\begin{align}
	\label{tempo.di.fuga}
	\| \gamma_{a, b}(t^*) - a\| = \overline{C} {\|a\|}^{1-\frac{\left(\gamma_s - s \mu\right)}{\alpha_s}}\,,\\
	\label{tempo.di.fuga.1}	\| \gamma_{a, b}(t) - a\| < \overline{C}{\|a\|}^{1-\frac{\left(\gamma_s - s \mu\right)}{\alpha_s}} \quad \forall t \in [0, t^*)\,.
	\end{align}
        By {construction of $\gamma_{a,b}$ one has}
        $$
	\gamma_{a,b}(t){=}a+\Pi_M(\gamma_{a,b}(t)-a) \in \cC \quad \forall t \in [0, t^*)\,.
          $$

Denote $u(t):=\Pi_{M_a}(\gamma_{a,b}(t)-a)$ and note that, by Lemma \ref{lemma.pm.pmom} and \eqref{tempo.di.fuga.1}, if $\|a\|$ is large enough we have
        \begin{equation}\label{vengo.anch.io.0}
        \|u(t)\| \leq 2 \overline{C} {\|a\|}^{1-\frac{\left(\gamma_s - s \mu\right)}{\alpha_s}}\,.
        \end{equation}           
        {This, taking $\|a\|$ large enough and using that $1-\frac{\left(\gamma_s - s \mu\right)}{\alpha_s}<1$, implies that $\tilde u:= \gamma_{ab}(t) -a$ satisfies $a + \Pi_{M_a} \tilde u = a+ u(t) \in \cC_e$.  Therefore,
        from \eqref{pm.small} and applying Lemma \ref{lem.piem.om.above} with $u$ replaced by $\tilde u$,}
        $\exists C^+$, depending only on $\gamma_s, \mu, \tC_s,
        \tD_s$, such that
	\begin{equation}
	\label{sup}
	\left\|\Pi_{\Momega}\omega(a+\Pi_{\Momega}(\gamma_{ab}(t)-a)
	)\right\| \leq C^+ {\|a\|}^{\mom-\gamma_s+s\mu}\ .
	\end{equation}

	We are now going to use steepness in the form
        \eqref{steep.r}. 
        Note that, by Lemma \ref{lemma.pm.pmom} and
        \eqref{tempo.di.fuga}, if $\|a\|$ is large enough {we also have
        \begin{equation}\label{vengo.anch.io}
        \frac{\overline{C}}{2}  {\|a\|}^{1-\frac{\left(\gamma_s - s \mu\right)}{\alpha_s}} \leq \|u(t^*)\| \leq 2 \overline{C} {\|a\|}^{1-\frac{\left(\gamma_s - s \mu\right)}{\alpha_s}}\,.
        \end{equation}}

        Take then
	$r=\left\| a \right\|$ and $\xi:= \| u(t^*)\|$ and, for $\eta \in [0, \csi]$
	let $t_\eta$ be the smallest time in $[0, t^*]$ such that $\|
	u(t_\eta)\| = \eta\,$; by \eqref{vengo.anch.io} one has $\xi< \tr
	\left\| a\right\|$ (with $\tr$ the quantity
	in Eq.\
	\eqref{steep}).
	Let $\bar t$ be the point realizing the maximum on $[0, \xi]$ of the
	quantity $\displaystyle{\|\Pi_{\Momega}\omega({a} +
		u(t_\eta))\|}$, then steepness ensures that
	\begin{equation}\label{is.steep}
	\begin{aligned}
	\|\Pi_{\Momega}\omega(\gamma_{ab}(\bar t))\|=\|\Pi_{\Momega}\omega({a} + u(\bar t))\| = \max_{\eta \in
		[0, \xi]}\,\left\|\Pi_{\Momega}\omega({a} +
	u(t_\eta))\right\|
	\\ \geq \tB_s \left\| a\right\|^{\mom-\alpha_s}   \xi^{\alpha_s} \geq
	\tB_s (2^{-1}\overline{C})^{\alpha_{s}} {\|a\|}^{\mom-\gamma_s + s\mu}\,.
	\end{aligned}
	\end{equation}
	But, {taking $\overline{C}$ large enough,} this contradicts Eq.\ \eqref{sup}, and so the conclusion follows.
\end{proof}

\subsection{Proof of Theorem \ref{teo.invarianza}}\label{6.17}

This subsection follows very closely the proof given in Subsection 5.1
of \cite{BLMres} for the convex case. We prove in detail only the
lemmas with a new proof and we make reference to \cite{BLMres} for the
others.

The next two lemmas ensure that, if the parameters $\tC_{j}, \tD_j$ are
suitably chosen, an extended block $E^{(s)}_{M, j}$ is separated from
every resonant zone associated to a lower dimensional {subgroup}
$M^\prime\,$ which is not contained in $M.$  The mechanism is that a point in $E_{M,j}^{(s)}$ which is close to the resonant zone associated to a different {subgroup} $M'$ would gain one more resonance relation, and this would create a contradiction with Definition \ref{def blocks}.
\begin{lemma}[Nonoverlapping of resonances] \label{lemma che bei blocchi}
	For all $s = 1, \dots d-1$ there exist positive constants $\bar \tR$, $\bar{\tC}_{s+1}$ and $\bar{\tD}_{s+1}$, depending only on $d, \tC_{s}, \tD_{s}, \mu, \gamma_{s}\,,$
	such that the following holds: suppose that $M$ and $M^\prime$
	are two distinct {pure subgroups} of respective  dimensions  $s$ and $s^\prime$ with $s^\prime \leq s$ and  $M^\prime \nsubseteq M.$ If
	$$
	\tC_{s+1} > \bar{\tC}_{s+1}\,,\quad  \tD_{s+1} >\bar{\tD}_{s+1}\,, \quad \tR > \bar{\tR}\,,
	$$
	then
	$$
	E^{(s)}_{M, j} \cap \cZ^{(s^\prime)}_{M^\prime} = \emptyset \quad \forall j \in \cJ_M\,.
	$$
\end{lemma}

For the proof see Lemma 5.6 of \cite{BLMres}

\begin{lemma}[Separation of
	resonances] \label{lemma speriamo sia vero}   There exist positive constants $\bar \tR$, $\tilde{\tC}_{s+1}$ and $\tilde{\tD}_{s+1}$ depending only on $d, \mu, \gamma_{s}, \tC_{s}, \tD_{s}$ such that, if
	$$
	\tC_{s+1} > \tilde{\tC}_{s+1}\,, \quad \tD_{s+1} >\tilde{\tD}_{s+1}, \quad \tR > \bar \tR\,,
	$$
	then the following holds true. Let $a \in E^{(s)}_{M, j}$ for
	some $M$ of dimension $s= 1, \dots, d-1$ and some $j \in \cJ_M$, and let $k^\prime \in \Z^d$
	be such that 
	\begin{gather*}
	\norm{k^\prime} \leq {\| a + \sigma k^\prime\|}^\mu\,,
	\end{gather*}
	for some $\sigma \in \{0, 1\}$.
	Then $\forall M^\prime \not\subset M$ s. t.  $s':=\dim M^\prime {\leq} s$ one
	has 
	$$
	a + k^\prime \notin \cZ^{(s')}_{M^\prime}\,.
	$$
\end{lemma}

For the proof see Lemma 5.7 of \cite{BLMres}

As a consequence of Lemma \ref{lemma che bei blocchi}, one has the
following lemma, whose proof is a small variant of the proof of Theorem
5.8 of \cite {BLMres}.

\begin{lemma}\label{lemma.partition}
	If the constants $\tR$, $\tC_1, \dots, \tC_d,$ $\tD_1, \dots, \tD_d$ are chosen as in Lemma \ref{lemma che bei blocchi}, then the extended blocks $\{E^{(s)}_{M, j}\}_{M \subset \Z^d\,, j \in \cJ_M}$ are a partition of $\Lambda.$
\end{lemma}
%
%

Then we have the following lemma, which is also a variant of Theorem 5.10 of
\cite{BLMres}, so we will omit its proof. 

\begin{lemma}\label{lemma.invarianza}
	If the constants $\tC_1, \dots, \tC_d,$ $\tD_1, \dots, \tD_d$ and $\tR$ are chosen as in Lemma \ref{lemma speriamo sia vero}, and $Z$ is an operator in normal form, then one has
	$$
	[\Pi_{E^{(s)}_{M, j}}, Z] = 0 \quad \forall M \subset \Z^d, \ \forall j \in \cJ_M\,.
	$$
\end{lemma}
{We are now ready to prove Theorem \ref{teo.invarianza}:
\begin{proof}[Proof of Theorem \ref{teo.invarianza}]
 By Lemma \ref{lemma.partition}, the blocks $\{E^{(s)}_{M,j}\}$ are a partition of $\Lambda$, and by Lemma \ref{lemma.invarianza} they are left invariant by any normal form operator $Z$. It remains to prove that they are dyadic. This follows from Item 3 of Lemma \ref{lemma.diametri}: let $a$ and $b$ respectively be the points of minimum and maximum norm in $E^{(s)}_{M,j}$. Then Lemma \ref{lemma.diametri} implies
 $$
 \|b\| \leq \|a\| + \|b-a\| \leq \|a\| + C\|a\|^{1- \gamma_{s+1}}\leq 2 \|a\|\,,
 $$
 where in the last passage we have used that $1 - \gamma_{s+1} < 1$, that $\|a\| \geq \tR/2$ by Remark \ref{rmk.res.big}, and we have possibly increased the value of $\tR$.
\end{proof}
}

\section{Proof of the results on the applications}

\subsection{Proof of the results on the anharmonic oscillator}
\label{oscilla}

First we recall the properties of the action variables for the classical
Hamiltonian system
\begin{equation}
\label{ana.classica}
h_0(x,\xi)=\frac{\left\|\xi\right\|^2}{2}+
\frac{\left\|x\right\|^{2\ell}}{2\ell}\,, \quad {x \in \R^2,\ \xi \in \R^2\,,} 
\end{equation}
which were studied in \cite{anarmonico}.
The first action is 
defined to be the angular momentum $a_2(x,\xi):=x_1\xi_2-x_2\xi_1$.
Following \cite{anarmonico}, in order to define the action $a_1$,
consider the polar coordinates $(r,\theta)$ in $\R^2$, and the
effective potential $V_{L}^*(r):=\frac{L^2}{2r^2}+\frac{r^{2\ell}}{2\ell}
$. 
For $L\not=0$ {and $E >\min_rV^*_L(r)$,} we preliminary define
\begin{equation}
\label{ar}
a_r=a_r(E,L):=\frac{\sqrt2}{\pi} \int_{r_m}^{r_M}\sqrt{ E- V^*_L(r)}dr\ ,
\end{equation}
where $0<r_m <r_M$ are the two solutions of the equation $
E-V^*_L(r)=0 $.

The cone $\cC$ is defined by
\begin{equation}
\label{pigreco}
\cC:=\left\{a\in\R^2\ ;\ a_1\geq 0\ \textrm{ if } a_2 \geq0\,,\quad  a_1 \geq |a_2|\ \textrm{ if } \ a_2< 0 \right\}\,.
\end{equation}

{The following lemma was proved in \cite{anarmonico}.  }

\begin{lemma}[Lemma 4.5 of \cite{anarmonico}] \label{azioni.osci} 
	The function
	\begin{equation}
	\label{a1}
	a_1(E,L):=\left\{
	\begin{matrix}
	a_r(E,L) & for & L>0
	\\
	a_r(E,L)-L & for & L< 0
	\end{matrix}
	\right.   
	\end{equation}
	has the following properties:
	\begin{itemize}
		\item[(1)]  {it extends to a complex analytic function
                  of $L$ and $E$ defined in a complex neighbourhood of
                  the set}
		\begin{equation}
		\label{defia1.1}
	{	\left\{(E,L)\in\R^2\ :\ \left|L\right|<\left(\frac{2\ell}{\ell+1}E\right)^{\frac{\ell+1}{2\ell}}\ ,\quad
		E>0\right\}\ ;}
		\end{equation}
		\item[(2)] the map $E\mapsto a_1(E,a_2)$ admits an inverse
		$E=h_0(a_1,a_2)$ which is analytic in the interior of $\cC$. Furthermore it is homogeneous of degree
		$\frac{2\ell}{\ell+1}$ as a function of $(a_1,a_2)$.
		\item[(3)] the function
		$a_1(x,\xi):=a_1(h_0(x,\xi),a_2(x,\xi))$ is quasihomogeneous of
		degree $\ell+1$, namely
		$$
		a_1(\lambda x,\lambda^l\xi)=\lambda^{\ell+1}a_1(x,\xi)\ ,\quad \forall
		\lambda>0 \ .
		$$
		\item[(4)] There exist  positive constants $C_1,C_2$ s.t.
		$$
		C_1\langle a\rangle \leq \tk_0\leq C_2 \langle a\rangle \ ,
		$$
{                with $\tk_0$ defined in \eqref{k0.ana}.}
	\end{itemize}
\end{lemma}
{Still one has to show that $h_0$ extends to a complex neighborhood of
$\cC$, that it is steep on such an extended domain and finally to
prove the existence of the quantum actions.}

{To prove analyticity remark that a point of the boundary of the cone
  corresponds to a circular orbit of the particle, which in turn is
  a minimum of the effective 1-d Hamiltonian
  $h^*(p_r,r):=\frac{p_r^2}{2}+V^*_L(r)$.   Let $r_c=r_c(L)$ be the radius of such a
  circular orbit and denote $\tilde r:=r-r_c$. By Vey theorem \cite{Vey}, the
  Birkhoff normal form of $h^*$ is convergent in a complex
  neighborhood of the circular orbit, which means that there exists a
  canonical transformation which conjugates $h^*$ to a function}
  $$
\tilde h^*\left(\frac{p_r^2+\tilde r^2}{2}\right)
$$ {analytic in a neighborhood of the origin. Keeping track of the the
dependence on $L$, one easily verifies that $\tilde h^*$ is analytic
in a neighbourhood of any point of the boundary of the cone.}

{Then steepness is obtained by applying Theorem \ref{steep!} of the
appendix. According to Theorem \ref{steep!} it is enough to find a
point where the Arnold determinant (defined in Eq. \eqref{arnold})
does not vanish. Such a point is an arbitrary point of the boundary of
the cone: it was proved in Lemma 7 of \cite{BFS18} that the Arnold
determinant vanishes identically only when either $V(r)=-\frac{k}{r}$
or $V(r)=\frac{1}{2}kr^2$, which are not the case of the anharmonic
oscillator.}

  To conclude the proof we have then to prove the existence of the
  quantum actions. This is granted by the following result by
  Colin de Verdi\`{e}re (see Theorem 3.1 and Theorem 3.2 of \cite
  {CdV2}).

\begin{theorem}
	\label{3.1cdv}There exist two commuting pseudodifferential
	operators $A_j\in \cA_1^1$ satisfying Assumption \ref{assumption.A}. In
	particular there exists
	$\kappa\in\left(\frac{\Z}{4}\right)^2$ s.t their joint spectrum
        is contained in $\cC\cap(\Z^2+\kappa)$ and
	there exists a symbol $h\in S^{\frac{2\ell}{\ell+1}}_{AN,1}$ s.t.
	\begin{equation}\label{ham.laplaciano.1}
	H_0=h(A)\ .
	\end{equation}
	Furthermore one has an asymptotic expansion
	$$
	h=h_0+l.o.t.
	$$ with $h_0$ the classical
        Hamiltonian of the anharmonic oscillator written in terms of
        the action variables.
\end{theorem}

Finally we have to regularize the Hamiltonian at the origin, but this
is simply done by modifying it by the addition of a function with
compact support, whose quantization is a smoothing operator.

\subsection{Proof of the results on rotation-invariant surfaces}\label{ruota}

First we recall the Hamiltonian of the classical geodesic flow in the coordinates $(\phi,\theta)$ namely
\begin{equation}
\label{classical.manif}
h_0=\frac{p_\theta^2}{2}+\frac{p_\phi^2}{2r(\theta)^2}\ .
\end{equation}

Again we put $
a_2:=p_\phi$ and, 
for $p_\phi\not=0$ we define
\begin{equation}
\label{azint}
a_1=a_1(E,p_\phi):=
\frac{1}{\pi}\int_{\theta_m}^{\theta_M}\sqrt{E-\frac{p_\phi^2}{r(\theta)^2}
} d\theta +\left|p_\phi\right|\ ,
\end{equation}
where $\theta_m$ and $\theta_M$ are the solutions of
$E-\frac{p_\phi^2}{r(\theta)^2}=0$. 
The cone $\cC$ is defined by
$$
\cC:=\left\{(a_1,a_2)\in\R^2\ :\ a_1\geq |a_2|\geq 0\right\}\ ,
$$
and Colin de Verdi\`{e}re \cite{CdV2} proved the analogue of Theorem \ref{azioni.osci}
for this case. The proof can then be concluded exactly in the same way
as for the anharmonic oscillator. An explicit computation shows that
the conditions for the nonvanishing of the Arnold determinant at the
boundary of the cone are $\beta_2\not=0$ and \eqref{steep.riv.2}.  

\subsection{Proof of results on Lie groups}\label{sec.lie.proofs}

In order to prove that the quantum actions $A_1, \dots, A_d$  defined
in \eqref{lie.action} satisfy Assumptions \ref{assumption.A}, in this section we use
the intrinsic formulation of pseudodifferential calculus introduced in
\cite{Fischer, ruzhansky_turunen}. We keep
the notations of Subsection \ref{liegroups}, and for any $\csi \in
\widehat{ G}$, we denote by
\begin{equation}\label{lambda.csi}
\lambda_\csi := \| \tw_\csi + \underline \tf \|^2 - \|\underline \tf\|^2\,
\end{equation}
the corresponding eigenvalue of the Laplace-Beltrami operator $-\Delta_g$.

Following \cite{Fischer}, we consider the following class of symbols:
\begin{definition}\label{simbolo.fischer}
	Given $m \in \R$ and $0 \leq \delta_1 \leq \delta_2 \leq 1$,
        we say that a map $\sigma : G \times \Rep(G) \ni (x,
        \csi) \mapsto \sigma(x, \csi) \in \cB(\cH_\xi)$ is a symbol of
        order $m$, and we write $\sigma \in {S^{m}_{F,\delta_1,\delta_2}}$, if for any $\csi \in \widehat{G}$ the map $x
          \mapsto \sigma (x, \csi)$ is smooth and $\forall \alpha,
          \beta \in \N$ and for any smooth differential operator
          $D_x^\alpha$ of order $\alpha$ and any $\tau = (\tau_1,
          \dots, \tau_\alpha)$ with $\tau_i \in \widehat{ G}$, there
          exists $C>0$ such that
	\begin{equation}\label{insondabile}
	\| D_x^\alpha \Delta^\beta_\tau \sigma(x, \csi)\|_{{\mathcal B}(\mathcal H_{\csi}^{\otimes\tau})} \leq  C (1 + \lambda_\csi)^{\frac{m -\delta_2 \beta + \delta_1 \alpha}{2}}\,,
	\end{equation}
	where $\Delta^\beta_\tau := \Delta_{\tau_1} \cdots \Delta_{\tau_\beta}$,
	\begin{equation}\label{incrementi}
	\Delta_{\tau_i} \sigma(x, \csi) := \sigma(x, \tau_i \otimes \csi) - \sigma(x, \uno_{{\mathcal H}_{\tau_i}}\otimes \csi)\,,
	\end{equation}
	and $\| \cdot \|_{{\mathcal B}(\mathcal H_{\csi}^{\otimes\tau})}$ is the operatorial norm on ${\mathcal H}_{\tau_1} \otimes \cdots \otimes {\mathcal H}_{\tau_\beta} \otimes {\mathcal H}_\csi\,.$
\end{definition}

With the above definition, one defines the quantization $\textrm{Op}(\sigma)$ of a symbol $\sigma$ as
\begin{equation}\label{quant.fischer}
\textrm{Op}(\sigma) \psi(x) = \sum_{\csi \in \widehat{ G}} d_\csi \textrm{Tr}\left(\csi(x) \sigma(x, \csi) \hat{\psi}_\csi\right)\,,
\end{equation}
where $d_\csi = \textrm{dim} {\mathcal H}_\csi$ and $\hat{\psi}_\csi$ is the $\csi$-th Fourier coefficient of $\psi$.
If there exists {$\sigma \in S^{m}_{F,\delta_1, \delta_2}$} such that $\Sigma = \textrm{Op}(\sigma)$, we write {$\Sigma \in \OPS_{F,\delta_1, \delta_2}^m$}.
\begin{remark}
	Assume that the symbol $\sigma$ does not depend on $x$ and that it has the form $\sigma(\xi) = m(\xi) \uno_{\cH_\xi}$ for some function $m$.
	Then $\textrm{Op}(\sigma)$ defined as in \eqref{quant.fischer} acts as a Fourier multiplier, which multiplies each frequency $\hat{\psi}_\csi$ by the factor $m(\xi)$. This is for instance the case of the Laplace-Beltrami operator $-\Delta_g$ and of the quantum actions $A_1, \dots, A_d$.
\end{remark}

The remarkable fact is that, as proved in
\cite{Fischer}, the pseudodifferential calculus constructed in this way
is equivalent to the pseudodifferential calculus constructed
considering $G$ as a manifold. Precisely the following result was proven in \cite{Fischer}:
\begin{theorem}\label{teo.fischer}
	Let $\delta_2 = 1 -\delta_1 =: \varrho$; then the class {$\OPS^{m}_{F,\delta_1, \delta_2}$} coincides with the class of pseudodifferential operators in the sense of H\"ormander with symbol {$S^{m}_{H,\varrho}$} on $G$.
\end{theorem}

By Theorem \ref{teo.fischer}, the proof that the operators $A_1, \dots, A_d$ defined in \eqref{lie.action} are actually pseudodifferential operators of order 1 reduces to the following:
\begin{lemma}\label{lemma.sono.s1}
	For $j = 1, \dots, d$ let $\sigma_{A_j}$ be defined as in \eqref{def.sigmap}, then {$\sigma_{A_j} \in S_{F,1, 0}^1$}.
\end{lemma}
\begin{proof}
	According to Section 3.2 of \cite{Fischer}, since the symbols $\sigma_{A_j}$ are independent of $x$,  it is sufficient to prove that for any $\beta \in \N$ and $\Delta_\tau^\beta = \Delta_{\tau_1} \cdots \Delta_{\tau_\beta}$ there exists $C=C_\tau>0$ such that
	\begin{equation}\label{no.x}
	\|\Delta^\beta_{\tau}\sigma_{A_j}(\csi)\|_{{\mathcal B}({\mathcal H}^{\otimes \tau})} \leq C_\tau (1 + \lambda_\csi)^{\frac{1-\beta}{2}} \quad \forall \csi \in \widehat{ G}\,.
	\end{equation}
	If $\beta=0$, one has
	\begin{equation}
	\|\sigma_{A_j}(\csi)\|_{{\mathcal B}({\mathcal H}^{\otimes \tau})} = \max_{\csi \in \widehat{ G}}\,\{\tw_\csi^j + 1\} \lesssim (1 + \lambda_{\csi})^{\frac{1}{2}}\,.
	\end{equation}
	Consider now $\beta = 1$: then one has to estimate $\forall \tau \in \widehat{ G} $
	$$
	\|\Delta_\tau \sigma_{A_j}(\csi)\|_{{\mathcal B}({\mathcal H}_{\tau} \otimes {\mathcal H}_\csi)} = \|\sigma_{A_j}(\tau \otimes \csi) - \sigma_{A_j}(\uno_{{\mathcal H}_\tau} \otimes \csi)\|_{{\mathcal B}({\mathcal H}_{\tau} \otimes {\mathcal H}_\csi)}\,.
	$$
	Now (see for instance \cite{fulton_harris}, Exercise 25.33), for any $\tau, \csi \in \widehat{ G}$, if $\tw_\csi$ is the highest weight of $\csi$, then the representation $\tau \otimes \csi$ is isomorphic to a finite direct sum of representations $\zeta_\mu$ with highest weight $\tw_\csi + \mu$, for any weight $\mu$ of the representation $\tau$. Thus one has
	\begin{equation}\label{van.geemen}
	\begin{aligned}
	\Delta_\tau \sigma_{A_j}(\csi) &= \bigoplus_{\mu\in \textrm{ weight of } \tau} \left((\tw_\csi + \mu)^j + 1 - \left(\tw_\csi^j + 1\right)\right) \uno_{{\mathcal H}_{\zeta_\mu}}\\
	&= \bigoplus_{\mu\in \textrm{ weight of } \tau} \mu^j\, \uno_{{\mathcal H}_{\zeta_\mu}}\,.
	\end{aligned}
	\end{equation}
	This implies
	$$
	\begin{aligned}
	\|\Delta_\tau \sigma_{A_j}(\csi)\|_{{\mathcal B}({\mathcal H}_{\tau} \otimes {\mathcal H}_\csi)} & = \max_{\mu\in \textrm{ weight of } \tau} | \mu^j| =: C_\tau\,,
	\end{aligned}
	$$
	which gives \eqref{no.x} with $\beta = 1$.
	Finally, if $\beta \geq 2$, by \eqref{van.geemen} one gets that $\Delta^\beta_{\tau} \sigma_{A_j} = 0\,,$
	thus \eqref{no.x} is trivially verified.
\end{proof}

\appendix

\section{The generalized commutator lemma}\label{a.commut}

In this section we apply the
commutator lemma proved in \cite{Ras12} to our pseudodifferential
setting. 

Given a multiindex $j=(j_1,...,j_d)$, {and $d$ selfadjoint pairwise
commuting operators $A_1,...,A_d$,} we will denote
$$
\operatorname{Ad}_{A}^j:=\operatorname{Ad}_{A_1}^{j_1}...\operatorname{Ad}_{A_d}^{j_d}\ .
$$
Then the following theorem holds:
\begin{theorem}[Theorem 3 of \cite{Ras12}]
	\label{commutator}
	Let $B\in\cB(\cH)$ be such that $\operatorname{Ad}_A^{j}B\in\cB(\cH)$
	for all multiindexes $j$ {with $|j|\leq n_0$}. Let $f\in S^m_1$, then, for all positive $t_1,t_2$ and all integers $n$
	fulfilling
	$$
	t_1+t_2+m<n+1\ ,{\quad n+1\leq n_0}
	$$
	one has 
	\begin{equation}
	\label{formula_commuti}
	\left[B;f(A)\right]=\sum_{|j|=1}^n\frac{1}{j!}\partial^jf(A)\operatorname{Ad}_A^j B+R_n(A,B)\,.
	\end{equation}
	Furthermore there exists a positive constant $C$ (which
	depends on $n$), independent of $A$ and $B$, s.t.
	$$
	\left\|R_n(A,B)\right\|_{\cB(\cH^{-t_2}, \cH^{t_1})} \leq C \sum_{|j|\leq n+1}
	\left\|\operatorname{Ad}_A^{j}B\right\|_{\cB(\cH)}\ .
	$$
\end{theorem}

\begin{corollary}
	\label{cor_commuti}
	Let $B\in\cA^{m'}_\rho$, let $f$ be as above and {let
        $A_1,...,A_d$ be quantum actions as in Assumption A}; let $N$ and $n$
	fulfill $N+m<\rho(n+1)$ then equation \eqref{formula_commuti} holds. Furthermore there exists $J$ and for any $s\in\R$ there exists a
	constant $C_{s,N}$ s.t.\ the following estimate holds
	\begin{equation}
	\label{resto_commuti}
	\left\|R_n(A,B)\right\|_{\cB(\cH^s;\cH^{s+N-m'})}\leq
	C_{s,N}\semi{m}JB\ .
	\end{equation}
\end{corollary}
To get the corollary, just apply Theorem \ref{commutator} to the
operator $\langle A\rangle^{-m'-n(1-\rho)+s}B\langle A\rangle^{-s}$ and take 
$t_2=0$, $t_1=N$.\qed

\section{Steepness in 2 degrees of freedom}\label{minima}

We will use here a condition equivalent to steepness introduced
in \cite{Nied06}.

\begin{theorem}[Niederman] \label{nie06} Let $h$ be a function real analytic
        in an open set $\cU\subset\R^d$. Then $h$ is steep on any
        compact set $\Sigma\subset \cU$ if and only if {$h$ has no
          critical points in $\cU$,} and its restriction $h\big|_{M}$
        to any affine subspace {$M\subset \R^d$} admits only isolated
        critical points.
\end{theorem}

If $d=2$, then, by homogeneity, one gets a condition easy to
verify. To state the corresponding result, we recall the definition of
the Arnold determinant $\cD$:
\begin{equation}
\label{arnold}
\cD:=\det\left[\begin{matrix}
\frac{\partial^2 h}{\partial a^2} & \left(\frac{\partial h}{\partial
	a}\right)^t
\\
\frac{\partial h}{\partial a}& 0
\end{matrix}\right] \ .
\end{equation}

\begin{theorem}
	\label{steep!}
	Let $\cC_e\subset \R^2$ be an open convex cone, and let $h:\cC_e\to\R$ be a real
	analytic function homogeneous of degree $\td$. Assume that there exists
	$\bar a\in\cC_e$ s.t.\ $\cD(\bar a)\not=0$, and that $h(a)>0$ $\forall
	a\in\cC_e$, then $h$ is steep in {any} compact subset of $\cC_e$. 
\end{theorem}
\proof {Step 1. We prove that $h$ has no critical points in
$\cC_e$. Indeed, if $\bar a$ is a critical point of $h$ then, $\forall
t>0$, also
$t\bar a$ is a critical point of $h$. It follows that $h$ is
constant along the line $t\bar a$, $t\in\R^+$, but one has
$h(t\bar a)=t^{\td}h(\bar a)$ which is not constant, since $h$ is
positive. }

{Step 2. We show that the set where $\cD(a)$ vanish is composed of
radial lines. Assume by contradiction that it vanishes on a curve
intersecting transversally (topologically) a radial line $\Delta$, then, since
$\cD$ is a homogeneous function, it vanishes on a whole open cone
containing $\Delta$, thus by analyticity it is identically zero, which
contradicts the assumptions. }

{Step 3. The restriction of $h$ to any non radial straight line
$\Delta$ is a nontrivial analytic function. Indeed, by step 2,
there exists at least one point $a_0\in\Delta$ with $\cD(a_0)\not=0$,
and recall that $\cD(a_0)\not =0$ implies that $h$ is quasiconvex at $a_0$.
Write $\Delta$ as the set of the points of the form $a=t
n+a_0$, with $n$ a unitary vector, one has
	$$
	h(tn+a_0)=h(a_0)+t dh(a_0)n+ \frac{1}{2}d^2h(a_0)(n,n)t^2+...
	$$
        but, if $dh(a_0)n\not =0$, then this is a nontrivial function of
	$t$, while, if $dh(a_0)n=0$ then by quasiconvexity one has
	$d^2h(a_0)(n,n)\not=0$, so also in this case this is a nontrivial
	function of $t$.}}
        
{Step 4. The restriction of $h$ to any radial line is a nontrivial
analytic function. One has
$$
h(t a_0)=t^\td h(a_0)\ ,
$$
which is a nontrivial function of $t$, since $h(a_0)\not=0$.} \qed

\section{{Further technical lemmas}}

In this appendix we collect a few technical results that we use to
prove Theorem \ref{teo.invarianza}.
\begin{remark}\label{ep.piccolo}
	{For $x >0$, $y>1$, $C>0$ and $0<a<1$, one has}
	$$
	x < C \left( y + x^{a}\right)\quad \Longrightarrow \quad \exists
        C'>0\ s.t.\  	x < C' y\,.
        $$
\end{remark}
\begin{lemma}\label{rmk.a.or.b}
	If $a, b \in \R^2$, $C>0$ and $0< \mu < 1$ are such that ${\|a\|, \|b\| \geq 1}$,
	$$
	\| a - b\| \leq C {\| b\|}^{\mu}\,,
	$$
	then one has
        \begin{align*}
	\| a - b\| \leq C {\| b\|}^{\mu} \lesssim_{\mu} C {\| a\|}^{\mu} + C \| a -b\|^{\mu}\,,
\\
	\| a - b\| \lesssim_{C,\mu} {\| a\|}^{\mu}\,.
        \end{align*}
\end{lemma}
{The proof is a simple computation and is omitted.}
\begin{lemma}\label{incrementini}
	Let $a, b \in \Lambda$ and $k, h, l \in \Z^d$ {be such that $a + k, b + l \neq 0$}. Suppose that there exist positive constants $C, D, R, \gamma$, $\mu$ and $\delta$  such that $\mu < 1-\delta$ and
	\begin{equation}\label{brutte}
	\begin{gathered}
	\left|\scala{\omega(a + k)}{h}\right| \leq C {\| a + k \|}^{\mom-\gamma} \| h\|\,, \\
	\| a + k \| \geq R\,, \quad \|h\| \leq D {\|a + k \|}^{\mu}\,, \quad \|k\| \leq D {\| a + k\|}^{\mu}\,, 
	\end{gathered}
	\end{equation}
	\begin{equation}\label{ancora.brutte}
	\| b -a\| \leq D {\| a \|}^{1 -\delta}\,, \quad \| l\| \leq D {\|  b + l\|}^{\mu}\,.
	\end{equation}
	Then there exist positive constants $C^+, D^+$, depending only on $C, D, \mom, \gamma, \mu, \delta$, such that
	\begin{equation}\label{belle}
	\left|\scala{\omega(b + l)}{h}\right| \leq C^+ {\| b + l\|}^{\mom-\min\{\gamma, \delta\}} \|h\|\,,\quad \|h\| \leq D^+ {\| b + l \|}^{\mu}\,. 
	\end{equation}
\end{lemma}

\begin{proof}
The second of \eqref{belle} is simply obtained by triangle
inequality and exploiting $\mu<1$ and $1-\delta<1$. To obtain the
first one write
\begin{equation}\label{scal.om}
	\left|\scala{\omega(b+l)}{h}\right| \leq \left|\scala{\omega(a + k)}{h}\right| + \int_{0}^{1}\left| \scala{\left(\frac{\partial \omega(a+ k + t \eta)}{\partial a} \eta \right)}{h}\right|\,dt\,,
	\end{equation}
	{ with $\eta := b + l - (a + k)$. Due to homogeneity of the function $\omega$, this also implies}
	\begin{equation}\label{conti.2}
	\begin{aligned}
	\left|\scala{\omega(b+l)}{h}\right|
	&\leq \left|\scala{\omega(a + k)}{h}\right|  + K \int_{0}^{1}\| a + k +t\eta\|^{\mom-1} \,dt \, \| \eta \| \|h\|\,,
	\end{aligned}
	\end{equation}
	with
	$$
	K := \sup_{\left\| \hat{a}\right\|=1} \left\{\left\| \frac{\partial \omega(\hat{a})}{\partial a}\right\| \right\}\,.
	$$
Then the proof is concluded using triangular inequality.
\end{proof}

\begin{lemma}\label{lemma.proiettori}
	Let $\omega_* \in \R^d$, $\varepsilon >0$ and $M \subset \R^d$ be such that
	\begin{equation}\label{lil}
	\| \Pi_{M} \omega_*\| \leq \varepsilon\,.
	\end{equation}
	Define $M':= \Pi_{\omega_*^\bot} M$. If $\varepsilon \leq \frac{1}{2} \| \omega_*\|$, then one has
	\begin{equation}\label{near}
	\| \Pi_{M'} - \Pi_M\| \leq 9 \varepsilon \| \omega_*\|^{-1}\,.
	\end{equation} 
\end{lemma}
\begin{proof}
	Define $w= \Pi_{M} \omega_*$ and $v = \Pi_{M^\bot} \omega_*$, then one has
	\begin{equation}\label{tutta}
	\omega_* = v + w\,, \quad v \in M^\bot\,, \quad w \in M\,,
	\end{equation}
	with
	\begin{equation}\label{solo}
	\| \omega_* - v\| \leq \varepsilon\,.
	\end{equation}
	Denote $M \oplus v:=M \oplus \operatorname{span} v$; 
        we first prove that
	\begin{equation}\label{somma}
	M \oplus v = M' \oplus \omega_*\,.
	\end{equation}
	By the very definition of $M'$ and $v$, one has
	$$
	M'\subset M\oplus v\ \Longrightarrow\ M'\oplus\omega_*\subset M \oplus
	\omega_* =M\oplus v\ .
	$$
To prove the opposite inclusion, remark that for any $m \in M$, one has $\Pi_{M'} m = \Pi_{\omega_*^\bot} m$
	$$
	m = \Pi_{\omega_*}m + \Pi_{\omega_*^\bot}m = \Pi_{\omega_*}m+\Pi_{M'}m   \in M' \oplus \Span\{\omega_*\}\,,
	$$
	so we have
	$$
	M\subset M'\oplus \omega_*\ \Longrightarrow\ M'\oplus\omega_*\supset M \oplus
	\omega_* =M\oplus v\ .
	$$
Since 
	$M\perp v=M'\perp \omega_*$, one has that, for any $u \in \R^d$
	\begin{equation}\label{pie}
	\Pi_M u = \Pi_{M\perp v} u- \frac{\scala{v}{u}}{\| v\|^2} v\,, \quad \Pi_{M'} u = \Pi_{M'\perp \omega_*} - \frac{\scala{\omega_*}{u}}{\| \omega_*\|^2} \omega_*\,.
	\end{equation}
	thus
	\begin{align*}
	\| \Pi_{M} u - \Pi_{M'} u \| & = \left\| \frac{\scala{v}{ u}}{\| v\|^2} v - \frac{\scala{\omega_*}{ u}}{\| \omega_*\|^2} \omega_* \right\|\,,
	\end{align*}
	{        from which the conclusion follows.}
\end{proof}

\bibliographystyle{alpha}

\end{document}